\documentclass[12pt]{amsart}
\usepackage{amsmath,amssymb,amsthm,textcomp,pst-plot,pstricks-add,mathscinet,mathtools}
\usepackage{amsfonts,graphicx}
\usepackage{color,amsaddr}
\usepackage[mathscr]{eucal}
\theoremstyle{definition}
\usepackage[american]{babel}
\usepackage{csquotes}

\numberwithin{equation}{section}
\newcommand{\ncom}{\newcommand}
\ncom{\nno}{\nonumber}
\ncom{\vone}{\vskip 2ex}
\ncom{\norm}{\|\;\;\|}
\ncom{\vspan}[1]{{{\rm\,span}\{ #1 \}}}
\ncom{\dm}[1]{ {\displaystyle{#1} } }
\ncom{\ri}[1]{{#1} \index{#1}}

\newtheorem{theorem}{\bf Theorem}[section]
\newtheorem{remark}{\bf Remark}[section]

\newtheorem{lemma}{Lemma}[section]
\newtheorem{corollary}{Corollary}[section]

\newtheorem{definition}{Definition}[section]
\newtheoremstyle
    {remarkstyle}
    {}
    {11pt}
    {}
    {}
    {\bfseries}
    {:}
    {     }
    {\thmname{#1} \thmnumber{#2} }

\theoremstyle{remarkstyle}

\def \R{{{\rm I{\!}\rm R}}}
\def \C{{{\rm I{\!\!\!}\rm C}}}

\def\R{{\mathbb R}}

\def\C{{\mathbb C}}

\def\E{{\mathbb E}}
\def\P{{\mathbb P}}
\def\Z{{\mathbb Z}}

\begin{document}
\title{ f\lowercase{ractional} N\lowercase{egative} B\lowercase{inomial and} P\lowercase{olya} P\lowercase{rocesses}}

\author{\small P. Vellaisamy and A. Maheshwari}
\address{\small Department of Mathematics,\\
 Indian Institute of Technology Bombay, Powai, Mumbai 400076, INDIA.}
 \email{pv@math.iitb.ac.in, aditya@math.iitb.ac.in}
 \thanks{The research of AM was supported by UGC, Govt. of India grant F. 2-2/98 (SA-1)}
\begin{abstract}
In this paper, we define a fractional negative binomial process (FNBP) by replacing the Poisson process by a fractional Poisson process (FPP) in the gamma subordinated form of the negative binomial process. First, it is shown that the one-dimensional distributions of the FPP are not infinitely divisible.  The long-range dependence of the FNBP, the short-range dependence of its increments and the infinite divisibility of the FPP and the FNBP are investigated. Also, the space fractional Polya process (SFPP) is defined by replacing the rate parameter $\lambda$ by a gamma random variable in the definition of the space fractional Poisson process. The properties of the FNBP and the SFPP and the connections to $pde$'$s$ governing the density of the FNBP and the SFPP are also investigated.
\end{abstract}

\maketitle

\vspace*{-0.7cm}
\section{Introduction}
\noindent The fractional generalizations of classical stochastic processes have received considerable attention by researchers in the recent years. These generalizations have found applications in several disciplines such as control theory, quantum physics, option pricing, actuarial science and reliability. For example, the fractional Poisson processes (FPPs) have been used recently in \cite{laskapp} to define a new family of quantum coherent states as well as the fractional generalization of Bell polynomials, Bell numbers and Stirling's numbers of the  second kind. Also, a new renewal risk model, which is non-stationary and has the long-range dependence property, is defined using the FPP in \cite{biardapp}. In this paper, we define a fractional generalization of the negative binomial process and a space fractional version of the Polya process. Quite recently, a fractional generalization of the negative binomial process is defined in \cite{beghin8may} and \cite{BegClau14}. We introduce here a different generalization of the negative binomial process. It is known that the negative binomial process can be viewed as a Poisson process time-changed by a gamma subordinator. Let $\alpha>0,~p>0$ and $\{\Gamma(t)\}_{t\geq0}$ be a gamma process, where $\Gamma(t)\sim G(\alpha,pt)$ which denotes the gamma distribution with scale parameter $\alpha^{-1}$ and shape parameter $pt$. Let
$$Q(t,\lambda)=N(\Gamma(t),\lambda),~~~t\geq0,$$ 
where $\{N(t,\lambda)\}_{t\geq0}$ is a Poisson process with intensity $\lambda>0$. Then $\{Q(t,\lambda)\}_{t\geq0}$ is called the negative binomial process and $Q(t,\lambda)\sim\text{NB}(pt,\eta)$, the negative binomial distribution with parameters $pt$ and $\eta=\lambda/(\alpha+\lambda)$ (see Section \ref{section-dist}). For $0<\beta<1$, let $\{D_{\beta}(t)\}_{t\geq0}$ be a $\beta$-stable subordinator with index $\beta$, and $\{E_{\beta}(t)\}_{t\geq0}$ be its (right-continuous) inverse stable subordinator defined by
\begin{equation}\label{hitting-time-process}
E_{\beta}(t) = \inf\{s>0: D_{\beta}(s)>t\},~~~t>0.
\end{equation}
\noindent A natural generalization of $\{Q(t,\lambda)\}_{t\geq0}$ is to consider 
\begin{equation*}
Q_{\beta}(t,\lambda)=N_{\beta}(\Gamma(t),\lambda),
\end{equation*} where $\{N_{\beta}(t,\lambda)\}_{t\geq0}$ is the FPP (see \cite{lask,mnv}). We call $\{Q_{\beta}(t,\lambda)\}_{t\geq0}$ the fractional negative binomial process (FNBP). We will show that this process is different from the FNBP discussed in \cite{beghin8may} and \cite{BegClau14}. It is known that the Polya process is obtained by replacing the parameter $\lambda$ by a gamma random variable in the definition of the Poisson process $\{N(t,\lambda)\}_{t\geq0}$. Let $\Gamma\sim G(\alpha,p)$ and  
$W^{\Gamma}(t)=N(t,\Gamma)$, where $\Gamma$ is independent of $\{N(t,\lambda)\}_{t\geq0}$. Then $\{W^{\Gamma}(t)\}_{t\geq0}$ is called the Polya process. However, a fractional
version of the Polya process has not been addressed in the literature before. Recently, in \cite{sfpp}, a space fractional Poisson process $\{\widetilde{N}_{\beta}(t,\lambda)\}_{t\geq0}$, where $\widetilde{N}_{\beta}(t,\lambda)=N(D_{\beta}(t),\lambda)$, is introduced and its properties are investigated. We here introduce, as a  fractional generalization of the Polya process, the space fractional Polya process (SFPP) defined by $\widetilde{W}^{\Gamma}_{\beta}(t)=\widetilde{N}_{\beta}(t,\Gamma)$ for $t\geq0$.

\vone \noindent  The paper is organized as follows. In Section \ref{secprelim}, some preliminary notations and results are stated. In Section \ref{secfpp}, we discuss the infinite divisibility of the FPP $\{N_{\beta}(t,\lambda)\}_{t\geq0}$ and also that of  $\{N(E_{\beta}^{\ast n}(t),\lambda)\}_{t\geq0}$, where $E_{\beta}^{\ast n}(t)$ is the $n$-iterated process of inverse stable subordinators. In Section \ref{secnbp}, we define the FNBP, compute its one-dimensional distributions and discuss their properties. It is shown that their one-dimensional distributions are not infinitely divisible, and they solve certain fractional $pde$'$s$. It is also shown that the FNBP exhibits the long-range dependence property and the increments of the FNBP possess the short-range dependence property. In Section \ref{secpp}, we define the SFPP and show that it has stationary increments and is stochastically continuous. However, it does not have independent increments and hence is not a L\'evy process. The fractional $pde$'$s$ 
governed by the SFPP with respect to both the variables $t$ and $p$ are also discussed.
\vone

\section{Preliminaries}\label{secprelim}
\noindent In this section, we introduce the notations and the results that will be used later. Let $\Z_{+}=\{0,1,\ldots,\}$ be the set of nonnegative integers.
\subsection{Some special functions}
We start with some special functions that will be required later.\\
(i) The Mittag-Leffler function $L_{\beta}(z)$ is defined as (see \cite{erde3})
\begin{equation}\label{Mittag-Leffler-function}
 L_{\beta}(z)=\sum\limits_{k=0}^{\infty}\frac{z^{k}}{\Gamma(1+\beta k)},\,\,\,\beta,z\in \C \text{ and Re}(\beta)>0.
\end{equation}
 \noindent (ii) The M-Wright function $M_{\beta}(z)$ (see \cite{gorenmain,mainardibook}) is defined as 
\begin{equation*}
 M_{\beta}(z)=\sum\limits_{n=0}^{\infty}\frac{(-z)^{n}}{n!\Gamma(-\beta n+(1-\beta))}=\frac{1}{\pi}\sum\limits_{n=1}^{\infty}\frac{(-z)^{n-1}}{(n-1)!}\Gamma(\beta n)\sin(\pi\beta n), ~z\in \mathbb{C},~0<\beta<1.
\end{equation*}
\noindent Let $p,q\in \Z_{+}\backslash\{0\}$. Also, for $0\leq i\leq p,~0\leq j\leq q$, let $a_{i},b_{j},z\in\mathbb{C}$.\\
(iii) Let $\alpha_{i}$ and $\beta_{j}$ be reals. The generalized Wright function is defined, under certain conditions (see \cite[p. 3]{wright}), as
 \begin{equation}\label{psifunction}
 {}_{_{p}}\psi_{_{q}}\equiv{}_{_{p}}\psi_{_{q}}\bigg[z~\bigg|\begin{matrix}
(a_{i},\alpha_{i})_{1,p}\\ (b_{j},\beta_{j})_{1,q}
\end{matrix}\bigg]=\sum_{k=0}^{\infty}\dfrac{\underset{i=1}{\overset{p}{\prod}}  \Gamma(a_{i}+\alpha_{i}k)}{\underset{j=1}{\overset{q}{\prod}} \Gamma(b_{j}+\beta_{j}k)}\bigg(\frac{z^{k}}{k!}\bigg).
\end{equation}
 (iv) Let $A_{i}$ and $B_{j}$ be positive reals. The $H$-function \cite[Section 1.2]{matsax} is defined in terms of the Mellin-Barnes type integral as 
\begin{equation}\label{mellin-H}
H^{m,n}_{p,q}(z)\equiv H^{m,n}_{p,q}\bigg[z~\bigg|\begin{matrix}
(a_{i},A_{i})_{1,p}\\ (b_{j},B_{j})_{1,q}
\end{matrix}\bigg]=\frac{1}{2\pi i}\int_{L}\chi_{p,q}^{m,n}(s)z^{-s}ds,
\end{equation}
where $z\neq 0$ and $z^{-s}=\exp[-s\{\ln|z|+i\arg z\}]$. Here, $\ln|z|$ represents the natural logarithm of $|z|$ and $\arg(z)$ is not necessarily the principal value.  Also, an empty product is interpreted as unity and
\begin{equation*}
 \chi_{p,q}^{m,n}(s)=\frac{\underset{i=1}{\overset{m}{\prod}}\Gamma(1-a_{i}-A_{i}s)\underset{j=1}{\overset{n}{\prod}}\Gamma(b_{j}+B_{j}s)}{\underset{i=m+1}{\overset{p}{\prod}} \Gamma(a_{i}+A_{i}s)\underset{j=n+1}{\overset{q}{\prod}}\Gamma(1-b_{j}+B_{j}s)},
\end{equation*}
where $m,n,p\text{ and }q$ are nonnegative integers such that $0\leq m\leq p,~1\leq n\leq q$ and
\begin{equation*}
 A_{i}(b_{j}+l)\neq B_{j}(a_{i}-k-1),
\end{equation*}
for $l,k\in\Z_{+},1\leq i\leq m\text{ and}~1\leq j\leq n$. The contour $L$ in \eqref{mellin-H} runs from $c-i\infty$ to $c+i\infty$ and separates the poles $s_{j,l}=-\big(\frac{b_{j}+l}{B_{j}}\big)$ of $\Gamma(b_{j}+B_{j}s)$ from the poles $w_{i,k}=\left(\frac{1-a_{i}+k}{A_{i}}\right)$ of $\Gamma(1-a_{i}-A_{i}s)$, where $1\leq i\leq m$ and $1\leq j\leq n$.\\
It is known that, under certain conditions (see \cite[eq. (5.2)]{wright}), the generalized Wright function ${}_{p}\psi_{q}$ given in \eqref{psifunction} satisfies
\begin{equation}\label{h-psi-relation}
{}_{_{p}}\psi_{_{q}}\bigg[z\bigg|\begin{matrix}
(a_{i},A_{i})_{(1,p)}\\ (b_{j},B_{j})_{(1,q)}
\end{matrix}\bigg]=H_{p,q+1}^{1,p}\bigg[-z\bigg|\begin{matrix}
(1-a_{i},A_{i})_{(1,p)}\\ (0,1),(1-b_{j},B_{j})_{(1,q)}
\end{matrix}\bigg].
\end{equation}
\subsection{Some elementary distributions}\label{section-dist}
  Let $\{N(t,\lambda)\}_{t\geq0}$ be a Poisson process with rate $\lambda>0$, so that
\begin{equation*}
p(n|t,\lambda)=\mathbb P[N(t,\lambda)=n]=\frac{(\lambda t)^{n}e^{-\lambda t}}{n!},~~~~n\in\Z_{+}. 
\end{equation*}
For $\alpha>0,~p>0$, let $\{\Gamma(t)\}_{t\geq0}$ be a gamma process, where $\Gamma(t)\sim G(\alpha,pt)$ with density
\begin{equation}\label{gammaden}
g(y|\alpha,pt)=\frac{\alpha^{pt}}{\Gamma{(pt)}}y^{pt-1}e^{-\alpha y},~~~~y>0.	 
\end{equation}
We say a random variable $X$ follows a negative binomial distribution with parameters $\alpha>0$ and $0<\eta<1$, denoted by $\text{NB}(\alpha,\eta)$, if
\begin{equation}\label{nbpmf}
 \P[X=n]=\binom{n+\alpha-1}{n}\eta^{n}(1-\eta)^{\alpha},\,\,\,\,\,n\in\Z_{+}.
\end{equation}
When $\alpha$ is a natural number, then $X$ denotes the number of successes before the $\alpha$-th failure in a sequence of Bernoulli trials with success probability $\eta$. \\
We say $X$ follows a logarithmic series distribution with parameter $\eta,$ denoted by $LS(\eta),$ if 
\begin{equation}
 \P[X=n] = \frac{-\eta^{n}}{n\ln(1-\eta)}, \quad n \in \Z_{+}\backslash\{0\}.
\end{equation}
 Let $\{D_{\beta}(t)\}_{t\geq0}$ be a $\beta$-stable subordinator. Then the density of $D_{\beta}(t)$ is (see \cite[eq. (4.7)]{gorenmain})
 \begin{equation}\label{stable-density}
  g_{_{\beta}}(x,t)=\beta tx^{-(\beta+1)}M_{\beta}(tx^{-\beta}), ~~~x>0.
 \end{equation}

 \noindent Let $\{E_{\beta}(t)\}_{t\geq0}$ be an inverse $\beta$-stable subordinator defined in \eqref{hitting-time-process}. Then the density of $E_{\beta}(t)$ is (see \cite[eq. (5.7)]{gorenmain})
 \begin{equation}\label{inverse-stable-density}
  h_{_{\beta}}(x,t)=t^{-\beta}M_{\beta}(t^{-\beta}x),~~~x>0.
 \end{equation}

\subsection{Some fractional derivatives}
Let $AC[a,b]$ be the space of functions $f$ which are absolutely continuous on $[a,b]$ and 
\begin{equation*}
 AC^{n}[a,b]=\left\{f:[a,b]\rightarrow \mathbb{R};\frac{d^{n-1}}{dt^{n-1}}f(t)\in AC[a,b]\right\},
\end{equation*}
where $AC^{1}[a,b]=AC[a,b]$.
\begin{definition}
 Let $m\in\Z_{+}\backslash\{0\}$ and $\nu\geq0$. If $f(t)\in AC^{m}[0,T]$, then the (left-hand) Riemann-Liouville (R-L) fractional derivative  $\partial_{t}^{\nu}f$ of $f$ (see \cite[Lemma 2.2]{KilSriTru06}) is defined by (with $\partial^{0}_{t}f=f$)
 \begin{equation}\label{rld}
  \partial_{t}^{\nu}f(t):= \begin{cases} 
     \hfill \dfrac{1}{\Gamma(m-\nu)}\dfrac{d^{m}}{dt^{m}}\displaystyle\int_{0}^{t}\dfrac{f(s)}{(t-s)^{\nu-m+1}}ds, \hfill    &m-1<\nu<m , \\ & \\
      \dfrac{d^{m}}{dt^{m}}f(t), \,\,\,\,\,\,\,\,\,\,\,  \nu=m. &
  \end{cases}
 \end{equation}
\end{definition}
\begin{definition}
 Let $m\in\Z_{+}\backslash\{0\}$ and $\nu\geq0$. If $f(t)\in AC^{m}[0,T]$, then the (left-hand) Caputo fractional derivative $D_{t}^{\nu}f$ of $f$ (see \cite[Theorem 2.1]{KilSriTru06}) is defined by (with $D^{0}_{t}f=f$)
 \begin{equation}\label{cd}
  D_{t}^{\nu}f(t):= \begin{cases} 
     \hfill \dfrac{1}{\Gamma(m-\nu)}\displaystyle\int_{0}^{t}\dfrac{f^{(m)}(s)}{(t-s)^{\nu-m+1}}ds, \hfill    &m-1<\nu<m , \\ &\\
      \dfrac{d^{m}}{dt^{m}}f(t), \,\,\,\,\,\,\,\,\,\,\, \nu=m.&
  \end{cases}
 \end{equation}
\end{definition}
\noindent The relation between the R-L fractional derivative and the Caputo fractional derivative is (see \cite[eq. (2.4.6)]{KilSriTru06})
\begin{equation*}
 \partial_{t}^{\nu}f(t)=D_{t}^{\nu}f(t)+\displaystyle\sum\limits_{k=0}^{m-1}\frac{t^{k-\nu}}{\Gamma(k-\nu+1)}f^{(k)}(0^{+}),
\end{equation*}
where $f^{(k)}(0^+):=\lim_{t\rightarrow 0^+} \frac{d^{k}}{dt^{k}}f(t)$.

\section{Fractional Poisson process}\label{secfpp}

\noindent Let $0<\beta\leq1$. The fractional Poisson process (FPP) $\{N_{\beta}(t,\lambda)\}_{t\geq0}$, which is a generalization of the Poisson process $\{N(t,\lambda)\}_{t\geq0}$, is defined to be a stochastic process for which $p_{\beta}(n|t,\lambda)=\mathbb{P}[N_{\beta}(t,\lambda)=n]$ satisfies (see \cite{lask,main,mnv})
\begin{flalign}
&&D^{\beta}_{t}p_{_{\beta}}(n|t,\lambda) &= -\lambda\left[ p_{_{\beta}}(n|t,\lambda)-p_{_{\beta}}(n-1|t,\lambda)\right],\,\,\,\text{for } n\geq1,& \label{fpp-definition}\\
&&D^{\beta}_{t}p_{_{\beta}}(0|t,\lambda) &= -\lambda p_{_{\beta}}(0|t,\lambda)\nonumber,&
\end{flalign}
 $\text{with  }p_{_{\beta}}(n|0,\lambda)=1\text{ if }n=0 \text{ and is zero if }n\geq1.$ Here, $D^{\beta}_{t}$ denotes the Caputo fractional derivative defined in \eqref{cd}.
\noindent The {\it pmf} $p_{_{\beta}}(n|t,\lambda)$ for the FPP is given by (see \cite{lask,mnv}) 
\begin{equation}\label{fppd}
 p_{_{\beta}}(n|t,\lambda)=\frac{(\lambda t^{\beta})^n}{n!}\sum_{k=0}^{\infty}\frac{(n+k)!}{k!}\frac{(-\lambda
t^{\beta})^k}{\Gamma(\beta(k+n)+1)}. \end{equation}
 
\noindent Note that equation (\ref{fppd}) can also be represented as
\begin{equation*}
p_{_{\beta}}(n|t,\lambda) = \frac{(\lambda t^{\beta})^{n}}{n!}{}_{_{1}}\psi_{_{1}}\bigg[-\lambda t^{\beta}\bigg|\begin{matrix}
(n+1,1)\\ (n\beta+1,\beta)
\end{matrix}\bigg],
\end{equation*}
using the generalized Wright function defined in \eqref{psifunction}.\\
 The mean and the variance of the FPP are given by (see \cite{lask})
\begin{align}
\E [N_{\beta}(t,\lambda)] &= qt^{\beta},\label{fppmean} \\
\mbox{Var}[N_{\beta}(t,\lambda)] &=qt^{\beta}\left[1+qt^{\beta}\left(\frac{\beta B(\beta, 1/2)}{2^{2\beta-1}}-1\right)\right],\label{fppvar}
\end{align}
where $q=\lambda/\Gamma(1+\beta)$ and $B(a,b)$ denotes the beta function. An alternative form for Var[$N_{\beta}(t,\lambda)$] is given in \cite[eq. (2.8)]{BegOrs09} as
\begin{equation}\label{alternative-fppvar}
 \mbox{Var}[N_{\beta}(t,\lambda)]=q t^{\beta}+\frac{(\lambda t^{\beta})^{2}}{\beta}\left(\frac{1}{\Gamma(2\beta)}-\frac{1}{\beta\Gamma^{2}(\beta)}\right).
\end{equation}
Note \eqref{alternative-fppvar} follows from \eqref{fppvar} using the Legendre's duplication formula (see \cite[p. 22]{askey})
\begin{equation*}
\Gamma(2a)\Gamma(1/2)=2^{2a-1}\Gamma(a)\Gamma(a+1/2),~a>0.
\end{equation*}
\noindent It is also known that (see \cite{mnv}) when $0<\beta<1,$
\begin{equation}\label{N-of-E-beta-t}
N_{\beta}(t,\lambda)\stackrel{d}{=}N(E_{\beta}(t),\lambda), 
\end{equation}
where $\{E_{\beta}(t)\}_{t\geq0}$ is the inverse $\beta$-stable subordinator and is independent of $\{N(t,\lambda)\}_{t\geq0}$.\\
First we establish an important property of the FPP.
\begin{theorem}\label{fpp-id}
Let $0<\beta<1$. The one-dimensional distributions of the FPP $\{N_{\beta}(t,\lambda)\}_{t\geq0}$ are not infinitely divisible ($i.d.$).
\end{theorem}
\begin{proof}
\noindent 
Since the sample paths of $\{D_{\beta}(t)\}_{t\geq0}$ are strictly increasing, the process $\{E_{\beta}(t)\}_{t\geq0}$ has continuous sample paths. Further,
\begin{equation*}
\P(E_{\beta}(t)\leq x) = \P(D_{\beta}(x)\geq t). 
\end{equation*}
\noindent It is well known that a $\beta$-stable process is self-similar with index $1/\beta$, that is,
\begin{equation*}
D_{\beta}(ct)\stackrel{d}=c^{1/\beta}D_{\beta}(t),~c>0. 
\end{equation*}
Hence,
\begin{equation*}
\begin{split}
 \P(E_{\beta}(ct)\leq x) = \P(D_{\beta}(x)\geq ct)&= \P\left(\tfrac{1}{c}D_{\beta}(x)\geq t\right)= \P\left(D_{\beta}\left(\tfrac{x}{c^{\beta}}\right)\geq t\right)\\
 &= \P\left(E_{\beta}(t)\leq \tfrac{x}{c^{\beta}}\right)= \P(c^{\beta}E_{\beta}(t)\leq x).
\end{split}
\end{equation*}
That is, 
\begin{equation}\label{self-similar-ht}
E_{\beta}(ct)\stackrel{d}=c^{\beta}E_{\beta}(t),
\end{equation}
showing that $E_{\beta}(t)$ is also self-similar with index $\beta.$\\
Observe now that $$N_{\beta}(t,\lambda) \stackrel{d}{=} N(E_{\beta}(t),\lambda) \stackrel{d}{=}N(t^{\beta}E_{\beta}(1),\lambda). $$ 
By the renewal theorem for the Poisson process,
\begin{equation*}
\lim_{t\rightarrow\infty}\frac{N(t,\lambda)}{t}=\frac{1}{\lambda},~~a.s. 
\end{equation*}
This implies, since $E_{\beta}(t)$ is independent of $\{N(t,\lambda)\}_{t\geq0}$,
\begin{align}\label{i.d.limit}
\lim_{t\rightarrow\infty}\frac{N(t^{\beta}E_{\beta}(1),\lambda)}{t^{\beta}} = E_{\beta}(1)\lim_{t\rightarrow\infty}\frac{N(t^{\beta}E_{\beta}(1),\lambda)}{t^{\beta}E_{\beta}(1)}= \frac{E_{\beta}(1)}{\lambda},~~ a.s.,
\end{align}
since $E_{\beta}(1)>0~$ $a.s.$
Hence, for $0<\beta<1,$
\begin{equation*}
\frac{N_{\beta}(t,\lambda)}{t^{\beta}}\stackrel{\mathscr{L}}\longrightarrow \frac{E_{\beta}(1)}{\lambda},
\end{equation*}
where $\stackrel{\mathscr{L}}\rightarrow$ denotes convergence in law.
Assume now that $N_{\beta}(t,\lambda)$ is $i.d.$ Then $N_{\beta}(t,\lambda)/t^{\beta}$ is also $i.d.$ for each $t$. Since the limit of a sequence of $i.d.$ random variables is also $i.d.$ (see \cite[Lemma 7.8]{sato}), it follows that $E_{\beta}(1)/\lambda$ or equivalently $E_{\beta}(1)$ is $i.d.$, which is a contradiction since $E_{\beta}(t)$ is not $i.d.$ for $t>0$ (see \cite{kumhit}). Hence, the result follows. 
\end{proof}
 \noindent
Let $\{E_{\beta_1}(t)\},\ldots,\{E_{\beta_n}(t)\}$ be independent inverse stable subordinators and $\beta=\beta_{1}\beta_{2}\ldots\beta_{n}$. Consider the $n$-iterated process $\{E_{\beta}^{*n}(t)\}$, where $E_{\beta}^{*n}(t)=E_{\beta_{1}}\circ E_{\beta_{2}}\circ\ldots\circ E_{\beta_{n}}(t)$ and for example $E_{\beta_{1}}\circ E_{\beta_{2}}(t)=E_{\beta_{1}}(E_{\beta_{2}}(t))$. By \cite[Remark 2.5]{kumhit}, we have that $E_{\beta}^{*n}(t)$ is not $i.d.$ We have the following result for the Poisson process with time change $E_{\beta}^{*n}(t)$. 
\begin{theorem}
The one-dimensional distributions of the subordinated Poisson process $\{N(E_{\beta}^{*n}(t),\lambda)\}_{t\geq0}$ are not $i.d.$
\end{theorem}
\begin{proof}
For some $c>0$ and using \eqref{self-similar-ht}, we have
\begin{equation*}
E_{\beta_1}(E_{\beta_2}(ct)) \stackrel{d}= E_{\beta_1}(c^{\beta_2}E_{\beta_2}(t)) \stackrel{d}= c^{\beta_1\beta_2}E_{\beta_1}(E_{\beta_2}(t)). 
\end{equation*}
Thus, in general, for $\beta=\beta_{1}\beta_{2}\ldots\beta_{n}$, we have $E_{\beta}^{*n}(ct) \stackrel{d}{=} c^{\beta}E_{\beta}^{*n}(t)$ and hence
\begin{equation*}
\frac{N(E_{\beta}^{*n}(t),\lambda)}{t^{\beta}}\stackrel{\mathscr{L}}\longrightarrow \frac{E_{\beta}^{*n}(1)}{\lambda}.
\end{equation*}
which is not $i.d.$ and hence the result follows.
\end{proof}
\section{Fractional Negative Binomial Process}\label{secnbp}
\subsection{Definition and properties}
\noindent Let $\{N(t,\lambda)\}_{t\geq0}$ be a Poisson process and $\{\Gamma(t)\}_{t\geq0}$ be a gamma subordinator, where $\Gamma(t)\sim G(\alpha,pt)$ defined in \eqref{gammaden}, and be independent of $\{N(t,\lambda)\}_{t\geq0}$. The negative binomial process $\{Q(t,\lambda)\}_{t\geq0}=\{N(\Gamma(t),\lambda)\}_{t\geq0}$ is a subordinated Poisson process (see \cite{fell,kozubo}) with 
\begin{align*}
 \mathbb P[Q(t,\lambda)=n]&=\delta(n|\alpha,pt,\lambda)=\frac{\alpha^{pt}\lambda^{n}}{n!\Gamma{(pt)}}\int_{0}^{\infty}y^{n+pt-1}e^{-y(\alpha+\lambda)}dy\nonumber\\
 &=\binom{n+pt-1}{n}\bigg(\frac{\alpha}{\alpha+\lambda}\bigg)^{pt}\bigg(\frac{\lambda}{\alpha+\lambda}\bigg)^{n}=\binom{n+pt-1}{n}\eta^{n}(1-\eta)^{pt},
\end{align*}
where $\eta=\lambda/(\alpha+\lambda)$. That is, $Q(t,\lambda)\sim \text{NB}(pt,\eta)$ for $t>0$, defined in \eqref{nbpmf}.
\begin{definition}
The fractional negative binomial process (FNBP) is defined as 
\begin{equation*}
\{Q_{\beta}(t,\lambda)\}=\{N_{\beta}(\Gamma(t),\lambda)\}, ~~~t\geq0, 
\end{equation*}
where $\{N_{\beta}(t,\lambda)\}_{t\geq0}$ is the FPP and is independent of $\{\Gamma(t)\}_{t\geq0}$. 
\end{definition}
\noindent Let $g(y|\alpha,pt)$ denote the {\it pdf} of $\Gamma(t)$, given in \eqref{gammaden}. Then,
\begin{align}
 \P[Q_{\beta}(t,\lambda&)=n]=\delta_{_{\beta}}(n|\alpha,pt,\lambda)=\int_{0}^{\infty}p_{_{\beta}}(n|y,\lambda)g(y|\alpha,pt)dy\nno\\ 
 &=\frac{\lambda^{n}}{n!}\sum\limits_{k=0}^{\infty}(-\lambda)^{k}\frac{(n+k)!}{k!}\frac{1}{\Gamma{(\beta(n+k)+1)}}\frac{\alpha^{pt}}{\Gamma{(pt)}}\int_{0}^{\infty}e^{-\alpha y}y^{(n+k)\beta+pt-1}dy \nonumber\\
 &=\frac{\lambda^{n}}{n!}\sum\limits_{k=0}^{\infty}(-\lambda)^{k}\frac{(n+k)!}{k!}\frac{1}{\Gamma{(\beta(n+k)+1)}}\frac{\alpha^{pt}}{\Gamma{(pt)}}\frac{\Gamma{((n+k)\beta+pt)}}{\alpha^{(n+k)\beta+pt}} \nonumber\\
 &=\bigg(\frac{\lambda}{\alpha^{\beta}}\bigg)^{n}\frac{1}{n!}\sum\limits_{k=0}^{\infty}\frac{(n+k)!}{k!}\frac{\Gamma{((n+k)\beta+pt)}}{\Gamma{(pt)}\Gamma{(\beta(n+k)+1)}}\bigg(\frac{-\lambda}{\alpha^{\beta}}\bigg)^{k}\nonumber\\
 &=\frac{1}{\Gamma{(pt)}n!}\bigg(\frac{\lambda}{\alpha^{\beta}}\bigg)^{n}{}_{_{2}}\psi_{_{1}}\bigg[\frac{-\lambda}{\alpha^{\beta}}\bigg|
 \begin{matrix}
(n+1,1),&(n\beta+pt,\beta)\\ (n\beta+1,\beta)
\end{matrix}\bigg].\label{fnbpmfc}
\end{align}
Assume $\left|\frac{-\lambda}{\alpha^{\beta}}\right|<1$. Then, by \cite[Theorem 1(b)]{wright} and with $\delta=1^{-1}\beta^{-\beta}\beta^{\beta}=1$,   $\Delta=\beta-\beta-1=-1$, the associated series of ${}_{_{2}}\psi_{_{1}}$ function in \eqref{fnbpmfc} converges. Thus, we have proved the following result.
\begin{theorem}
Let $0<\beta\leq1$, $0<\lambda<\alpha^{\beta}$, where $\alpha>0$. Then the FNBP $\{Q_{\beta}(t,\lambda)\}_{t\geq0}$ has the one-dimensional distributions
 \begin{align}
  \delta_{_{\beta}}(n|\alpha,pt,\lambda)&=\frac{1}{\Gamma{(pt)}n!}\bigg(\frac{\lambda}{\alpha^{\beta}}\bigg)^{n}{}_{_{2}}\psi_{_{1}}\bigg[\frac{-\lambda}{\alpha^{\beta}}\bigg|\begin{matrix}
(n+1,1),&(n\beta+pt,\beta)\\ (n\beta+1,\beta)
\end{matrix}\bigg]\label{fnbpmf}\\
&=\frac{1}{\Gamma{(pt)}n!}\bigg(\frac{\lambda}{\alpha^{\beta}}\bigg)^{n}H_{2,2}^{1,2}\bigg[\frac{\lambda}{\alpha^{\beta}}\bigg|\begin{matrix}
(-n,1),&(1-n\beta-pt,\beta)\\ (0,1),&(-n\beta,\beta)
\end{matrix}\bigg],~n\in\mathbb{Z}_{+},\label{fnbpmfh}
 \end{align}
 where $H^{1,2}_{2,2}$ is the $H$-function defined in \eqref{mellin-H}.
 \end{theorem}

\noindent When $\beta =1$, we can see that $\delta_{_{1}}(n|\alpha,pt,\lambda)$ reduces to the {\it pmf} of NB$(pt,\eta)$ distribution. We next show that $\delta_{_{\beta}}(n|\alpha,pt,\lambda)$ is indeed a {\it pmf} for $0<\beta<1$ also. Note that
\begin{align*}
\displaystyle \sum_{n=0}^{\infty}\delta_{_{\beta}}(n|\alpha,pt,\lambda) &= \displaystyle\sum_{n=0}^{\infty}\displaystyle \left(\frac{\lambda}{\alpha^{\beta}}\right)^n\sum_{k=0}^{\infty}\binom{n+k}{k}\frac{\Gamma((n+k)\beta+pt)}{\Gamma(pt)\Gamma((n+k)\beta+1)}\left(\frac{-\lambda}{\alpha^{\beta}}\right)^k\\
& = \displaystyle \sum_{n=0}^{\infty} \frac{(\lambda/\alpha^{\beta})^n}{n!\Gamma(pt)}\sum_{k=n}^{\infty}
\frac{k!}{(k-n)!}\frac{\Gamma(k\beta+pt)}{\Gamma(k\beta +1)}\left(\frac{-\lambda}{\alpha^{\beta}}\right)^{k-n}\\
&=\frac{1}{\Gamma(pt)}\sum_{k=0}^{\infty}\frac{\Gamma(k\beta+pt)}{\Gamma(k\beta +1)}\sum_{n=0}^{k}\binom{k}{n}\left(\frac{\lambda}{\alpha^{\beta}}\right)^n\left(\frac{-\lambda}{\alpha^{\beta}}\right)^{k-n}
 =1,
\end{align*}
since only the term corresponding to $k=0$ remains.
\begin{remark}
Let $0<\alpha^{\beta}<\lambda$. Then, using the representation given in \eqref{N-of-E-beta-t} and from \eqref{inverse-stable-density}, we also obtain
\begin{align*}
\delta_{_{\beta}}(n|\alpha,pt,\lambda) &= \int_{0}^{\infty}\int_{0}^{\infty}p(n|x,\lambda)h_{_{\beta}}(x,y)g(y|\alpha,pt)dxdy \nno\\
&= \frac{\alpha^{\beta}}{\lambda\Gamma(pt)}\sum_{k=0}^{\infty}\binom{n+k}{n}\frac{\Gamma(pt-\beta-\beta k)}{\Gamma(-\beta k+(1-\beta))}\left(-\frac{\alpha^{\beta}}{\lambda}\right)^{k}\\
&=\frac{\alpha^{\beta}}{n!\lambda\Gamma(pt)}{}_{_{2}}\psi_{_{1}}\left[\frac{-\alpha^{\beta}}{\lambda}\bigg|\begin{matrix}
(n+1,1),&(pt-\beta,-\beta)\\ (1-\beta,-\beta)
\end{matrix}\right].
\end{align*}
\end{remark} 
\noindent Let us denote henceforth $q=\lambda/\Gamma(1+\beta)$, $d_{1}=2\lambda^{2}/\Gamma(2\beta+1)$ and $d_{2}=\beta q^{2}B(\beta, 1+\beta)$.
\begin{theorem}\label{Theorem-fnbp-dist-properties} The mean, variance and autocovariance functions of the FNBP $\{Q_{\beta}(t,\lambda)\}_{t\geq0}$ are given by\\
  (i) $~~~~~\E[Q_{\beta}(t,\lambda)] = q\dfrac{\Gamma(pt+\beta)}{\alpha^{\beta}\Gamma(pt)}=q\E[\Gamma^{\beta}(t)]\sim q\left(\dfrac{pt}{\alpha}\right)^{\beta}= \left(\dfrac{p}{\alpha}\right)^{\beta}\mathbb{E}[N_{\beta}(t,\lambda)]$, for large $t$,\\
  (ii) $~~~~\mbox{Var}[Q_{\beta}(t,\lambda)]=\dfrac{q\Gamma(pt+\beta)}{\alpha^{\beta}\Gamma(pt)}\left(1-\dfrac{q\Gamma(pt+\beta)}{\alpha^{\beta}\Gamma(pt)}\right)+\dfrac{d_{1}\Gamma(pt+2\beta)}{\alpha^{2\beta}\Gamma(pt)}$,\\
  \begin{flalign}
\text{(iii) }~~~~\text{Cov}[Q_{\beta}(s,\lambda),Q_{\beta}(t,\lambda)]&=q\frac{\Gamma(ps+\beta)}{\alpha^{\beta}\Gamma(ps)}+ d_{2}\frac{\Gamma(ps+2\beta)}{\alpha^{2\beta}\Gamma(ps)} &&\nonumber&\\-q^{2}\frac{\Gamma(ps+\beta)}{\alpha^{2\beta}\Gamma(ps)}&\frac{\Gamma(pt+\beta)}{\Gamma(pt)}
 +q^{2}\beta \mathbb{E}[\Gamma^{2\beta}(t)B(\beta,1+\beta;\Gamma(s)/\Gamma(t))].\nonumber&
\end{flalign}
\end{theorem}

\begin{proof}For simplicity, the parameter $\lambda$ is suppressed in $\{N_{\beta}(t,\lambda)\}_{t\geq0}$ and $\{Q_{\beta}(t,\lambda)\}_{t\geq0}$, when no confusion arises. Using a conditioning argument and the equation \eqref{fppmean}, we get
\begin{equation}\label{mean-fnbp}
\E[Q_{\beta}(t)] = \frac{q}{\alpha^{\beta}}\frac{\Gamma(pt+\beta)}{\Gamma(pt)}=q\E[\Gamma^{\beta}(t)]. 
\end{equation}

\noindent By Stirling's formula, $\big(\Gamma(pt+\beta)/\Gamma(pt)\big)\sim (pt)^{\beta}$ for large $t$, and so we get $$\E [N_{\beta}(\Gamma(t))]\sim \left(\tfrac{p}{\alpha}\right)^{\beta}qt^{\beta}=\left(\tfrac{p}{\alpha}\right)^{\beta}\E N_{\beta}(t,\lambda)=\mathbb{E}[\Gamma^{\beta}(1)]\mathbb{E}[N_{\beta}(t,\lambda)],$$
which proves Part (i). Using \eqref{fppmean} and \eqref{alternative-fppvar}, we get
\begin{align}
 \mbox{Var}[Q_{\beta}(t)]&=\text{Var}\big[\E[N_{\beta}(\Gamma(t))|\Gamma(t)]\big]+\E\big[\text{Var}[N_{\beta}(\Gamma(t))|\Gamma(t)]\big]\nonumber\\ 
 &=q\E[\Gamma^{\beta}(t)]\big(1-q\E[\Gamma^{\beta}(t)]\big)+d_{1}\E[\Gamma^{2\beta}(t)]\label{var-fnbp-1}\\
 &=\frac{q\Gamma(pt+\beta)}{\alpha^{\beta}\Gamma(pt)}\left(1-\frac{q\Gamma(pt+\beta)}{\alpha^{\beta}\Gamma(pt)}\right)+\frac{d_{1}\Gamma(pt+2\beta)}{\alpha^{2\beta}\Gamma(pt)}\nonumber\nonumber.
\end{align}
 Using the result 
 \begin{equation}\label{gamma-beta-approx}
 \mathbb{E}[\Gamma^{l}(t)]=\frac{1}{\alpha^{l}}\frac{\Gamma(pt+l)}{\Gamma(pt)},~~l>0,
 \end{equation}
Part (ii) follows. Next, we get from \cite[eq. (14)]{LRD2014},
  \begin{equation*}
   \text{Cov}[N_{\beta}(s),N_{\beta}(t)]=qs^{\beta}+ d_{2}s^{2\beta}+ q^{2}[\beta t^{2\beta}B(\beta,1+\beta;s/t)-(st)^{\beta}],~~~0<s\leq t,
  \end{equation*}
where  $B(a,b;x)=\int_{0}^{x}t^{a-1}(1-t)^{b-1}dt,~0<x<1$, is the incomplete beta function. Hence, from \eqref{fppmean}, 
\begin{equation}\label{efppsfppt}
\mathbb{E}[N_{\beta}(s)N_{\beta}(t)]=qs^{\beta}+d_{2}s^{2\beta}+ q^{2}\beta\left[t^{2\beta}B(\beta,1+\beta;s/t)\right],
\end{equation}
which leads to
\begin{align}
 \mathbb{E}[Q_{\beta}(s)Q_{\beta}(t)]&=\mathbb{E}\left[\mathbb{E}[N_{\beta}(\Gamma(s))N_{\beta}(\Gamma(t))|\Gamma(s),\Gamma(t)]\right]\nonumber\\
 &=q\mathbb{E}[\Gamma^{\beta}(s)]+ d_{2}\mathbb{E}[\Gamma^{2\beta}(s)]+\beta q^{2} \mathbb{E}[\Gamma^{2\beta}(t)B(\beta,1+\beta;\Gamma(s)/\Gamma(t))],\label{bivariate-fnbfp}
\end{align}
Hence, using \eqref{mean-fnbp}, we get
\begin{align}
 \text{Cov}[Q_{\beta}(s),Q_{\beta}(t)]&=q\mathbb{E}[\Gamma^{\beta}(s)]+ d_{2}\mathbb{E}[\Gamma^{2\beta}(s)]-q^{2}\mathbb{E}[\Gamma^{\beta}(s)]\mathbb{E}[\Gamma^{\beta}(t)]\nonumber\\
 &+\beta q^{2} \mathbb{E}[\Gamma^{2\beta}(t)B(\beta,1+\beta;\Gamma(s)/\Gamma(t))]\label{autocovariance-fnbfp}.
\end{align}
Using \eqref{gamma-beta-approx}, Part (iii) follows.
\end{proof}
\begin{remark}
 \noindent A stochastic process $\{X(t)\}_{t\geq0}$ is said be overdispersed if $\text{Var}[X(t)]-\mathbb{E}[X(t)]>0$ for all $t\geq0$ (see \cite{BegClau14}). Now for the FNBP
\begin{align*}
\text{Var}[Q_{\beta}(t)]-\mathbb{E}[Q_{\beta}(t)]&=\frac{d_{1}\Gamma(pt+2\beta)}{\alpha^{2\beta}\Gamma(pt)}-\left(\frac{q\Gamma(pt+\beta)}{\alpha^{\beta}\Gamma(pt)}\right)^{2}\\
&=\frac{\lambda^{2}}{\beta}\left(\frac{\mathbb{E}[\Gamma^{2\beta}(t)]}{\Gamma(2\beta)}-\frac{\left(\mathbb{E}[\Gamma^{\beta}(t)]\right)^{2}}{\beta\Gamma^{2}(\beta)}\right)\\
&\geq \left(\lambda\mathbb{E}[\Gamma^{\beta}(t)]\right)^{2}Z(\beta),~~(\because \mathbb{E}[\Gamma^{2\beta}(t)]\geq (\mathbb{E}[\Gamma^{\beta}(t)])^{2})
 \end{align*}
where $Z(\beta)=\frac{1}{\beta}\left(\frac{1}{\Gamma(2\beta)}-\frac{1}{\beta\Gamma^{2}(\beta)}\right)>0$ for all $\beta\in(0,1)$ (see \cite[Section 3.1]{BegClau14}). Hence, the FNBP exhibits overdispersion.
\end{remark}
\noindent The next result shows that the FNBP is not $i.d.$
\begin{theorem}\label{fnb-id}
The one-dimensional distributions of the FNBP $\{Q_{\beta}(t,\lambda)\}_{t\geq0}$ are not $i.d.$ 
\begin{proof}
 Since $E_{\beta}(t)\stackrel{d}{=}t^{\beta} E_{\beta}(1)$, we have  $$Q_{\beta}(t,\lambda)=N(E_{\beta}(\Gamma(t)),\lambda)\stackrel{d}= N(\Gamma^{\beta}(t)E_{\beta}(1),\lambda).$$  Using \eqref{i.d.limit},
 \begin{equation*}
 \begin{split}
\lim_{t\rightarrow\infty}\frac{N(\Gamma^{\beta}(t)E_{\beta}(1),\lambda)}{t^{\beta}} &= \lim_{t\rightarrow\infty}\frac{N(\Gamma^{\beta}(t)E_{\beta}(1),\lambda)}{\Gamma^{\beta}(t)}\left(\frac{\Gamma(t)}{t}\right)^{\beta}\\
&= \frac{E_{\beta}(1)}{\lambda}(\E\Gamma(1))^{\beta} = \frac{E_{\beta}(1)}{\lambda}\left(\frac{p}{\alpha}\right)^{\beta}~ a.s.,
\end{split}
 \end{equation*}
 since $\Gamma(t)\rightarrow\infty$ and $\Gamma(t)/t\rightarrow \E\Gamma(1),~a.s.$, as $t\rightarrow\infty.$ The result follows by contradiction since $E_{\beta}(1)$ is not $i.d.$ 
\end{proof}
\end{theorem}
\begin{remark}(i) In fact, the above result can be generalized for any subordinator $T(t)$, with $\mathbb{E}[T(1)]<\infty$. For a subordinator, the SLLN for L\'evy processes yields 
$\lim_{t\rightarrow \infty}$ $T(t)/t= \E[T(1)]~a.s$. Thus, $\frac{N_{\beta}(T(t),\lambda)}{t^{\beta}}\stackrel{\mathscr{L}}\longrightarrow \frac{E_{\beta}(1)}{\lambda}(\E[ T(1)])^{\beta}$ which is not $i.d.$ \\
\noindent(ii) Since
\begin{equation*}
\frac{N(E_{\beta}^{*n}(\Gamma(t)),\lambda)}{t^{\beta_1\beta_2\cdots \beta_n}}\stackrel{\mathscr{L}}\longrightarrow  \frac{E_{\beta}^{*n}(1)}{\lambda}(\E [\Gamma(1)])^{\beta_1\beta_2\cdots\beta_n}, ~\text{as }t\rightarrow\infty, 
\end{equation*}
it follows that the distributions of $\{N(E_{\beta}^{*n}(\Gamma(t)),\lambda)\}_{t\geq0}$ are also not $i.d.$ 
\end{remark}
\noindent We next present a formal definition of the long-range dependence (LRD) property and the short-range dependence (SRD) property.\\
Let $s>0$ be fixed and $t>s$. Suppose a stochastic process $\{X(t)\}_{t\geq0}$ has the correlation function Corr$(X(s),X(t))$ which behaves like $t^{-d}$ for large $t$ and some $d>0$. We say $\{X(t)\}_{t\geq0}$ has the LRD property if $d\in(0,1)$  and has the SRD property if $d\in(1,2)$ (see \cite{ovi-lrd}).
\begin{lemma}
Let $a\in\R$ and $b\leq 1$. For fixed $s,0\leq s<t$, the asymptotic expansion of $\E[\Gamma^{a}(s)\Gamma^{b}(t)]$, as $t$ tends to infinity,  is given by
\begin{align}
\E[\Gamma^{a}(s)\Gamma^{b}(t)]& \sim\E\left[\Gamma^{a}(s)\right]\E\left[\Gamma^{b}(t-s)\right]+b\E\left[\Gamma^{a+1}(s)\right]\E\left[\Gamma^{b-1}(t-s)\right]\label{gamma-beta-joint}.
\end{align}
\begin{proof}
Since $\Gamma(t)\rightarrow\infty,~a.s.$, we have
\begin{align}
\E\left[\Gamma^{a}(s)\Gamma^{b}(t)\right]&=\E\left[\Gamma^{a}(s)(\Gamma(t)-\Gamma(s))^{b}\left(\frac{\Gamma(t)}{\Gamma(t)-\Gamma(s)}\right)^{b}\right]\nonumber\\
&=\E\left[\Gamma^{a}(s)(\Gamma(t)-\Gamma(s))^{b}\left(1-\frac{\Gamma(s)}{\Gamma(t)}\right)^{-b}\right]\nonumber\\
&\sim\E\left[\Gamma^{a}(s)(\Gamma(t)-\Gamma(s))^{b}\left(1+b\frac{\Gamma(s)}{\Gamma(t)}\right)\right]\nonumber\\
&=\E\left[\Gamma^{a}(s)(\Gamma(t)-\Gamma(s))^{b}\right]+b\E\left[\Gamma^{a+1}(s)\frac{(\Gamma(t)-\Gamma(s))^{b}}{\Gamma(t)}\right].\nonumber
\end{align}
Since the gamma process $\{\Gamma(t)\}_{t\geq0}$ has stationary and independent increments, we have
\begin{align}
\E\left[\Gamma^{a}(s)\Gamma^{b}(t)\right]&\sim\E\left[\Gamma^{a}(s)\right]\E\left[\Gamma^{b}(t-s)\right]+b\E\left[\Gamma^{a+1}(s)\Gamma^{b-1}(t)\left(1-\frac{\Gamma(s)}{\Gamma(t)}\right)^{b}\right]\nonumber\\
&\sim\E\left[\Gamma^{a}(s)\right]\E\left[\Gamma^{b}(t-s)\right]+b\E\left[\Gamma^{a+1}(s)\Gamma^{b-1}(t)\right]-b^{2}\E\left[\Gamma^{a+2}(s)\Gamma^{b-2}(t)\right]\nonumber\\
&\sim\E\left[\Gamma^{a}(s)\right]\E\left[\Gamma^{b}(t-s)\right]+b\E\left[\Gamma^{a+1}(s)\Gamma^{b-1}(t)\right], \text{ for large }t\label{gamma-beta-joint-1}.
\end{align}
Applying the relation in \eqref{gamma-beta-joint-1} to the second term in the right-hand side of \eqref{gamma-beta-joint-1},
\begin{align}
\E\left[\Gamma^{a}(s)\Gamma^{b}(t)\right]&\sim\E\left[\Gamma^{a}(s)\right]\E\left[\Gamma^{b}(t-s)\right]+b\E\left[\Gamma^{a+1}(s)\right]\E\left[\Gamma^{b-1}(t-s)\right]\nonumber.\qedhere
\end{align}
\end{proof}
\end{lemma}

\begin{theorem}
 The FNBP $\{Q_{\beta}(t,\lambda)\}_{t\geq 0}$ has the LRD property.
 \begin{proof}\noindent First note that, by Stirling's approximation,
 \begin{equation}\label{stirlings-appx}
  \mathbb{E}[\Gamma^{\beta}(t)]=\frac{1}{\alpha^{\beta}}\frac{\Gamma(pt+\beta)}{\Gamma(pt)}\sim\left(\frac{pt}{\alpha}\right)^{\beta},\text{ for large }t.
 \end{equation}Consider the last term of $\text{Cov}[Q_{\beta}(s),Q_{\beta}(t)]$ given in \eqref{autocovariance-fnbfp}, namely,
 \begin{equation}\label{autocovariance-last-summand}
  \beta q^{2} \mathbb{E}\Big[\Gamma^{2\beta}(t)B(\beta,1+\beta;\Gamma(s)/\Gamma(t))\Big].
 \end{equation}
Using now the Taylor expansion of $B(\beta,1+\beta;s/t)$ at $t=\infty$ (see \cite[p. 8]{LRD2014}), we get
\begin{align}
\beta t^{2\beta}B(\beta,1+\beta;s/t)&=\beta  t^{2\beta}\left[\frac{1}{\beta}\left(\frac{s}{t}\right)^{\beta}-\frac{\beta}{1+\beta}\left(\frac{s}{t}\right)^{\beta+1}+O\left(\left(\frac{s}{t}\right)^{\beta+2}\right)\right]\nonumber\\
 &=(st)^{\beta}+O\left(t^{\beta-1}\right)\label{Taylor-expansion-1}.
\end{align}
Using \eqref{Taylor-expansion-1}, we get for large $t$,
\begin{align}
   q^{2}\beta \mathbb{E}[\Gamma^{2\beta}(t)B(\beta,1+\beta;\Gamma(s)/\Gamma(t))]&\sim q^{2}\mathbb{E}[\Gamma^{\beta}(s)\Gamma^{\beta}(t)]\nonumber\\
   &\sim q^{2}\E[\Gamma^{\beta}(s)]\E[\Gamma^{\beta}(t-s)]\label{autocovariance-last-summand-1}~~~(\text{using }\eqref{gamma-beta-joint}).
\end{align}

\noindent Using \eqref{stirlings-appx} and \eqref{autocovariance-last-summand-1}, \eqref{autocovariance-fnbfp} becomes for large $t$,
\begin{align}
 \text{Cov}[Q_{\beta}(s),Q_{\beta}(t)]&\sim q\mathbb{E}[\Gamma^{\beta}(s)]+d_{2}\mathbb{E}[\Gamma^{2\beta}(s)]\nonumber\\
 &~~~~~~~-q^{2}\mathbb{E}[\Gamma^{\beta}(s)]\left(\frac{pt}{\alpha}\right)^{\beta}+q^{2}\mathbb{E}[\Gamma^{\beta}(s)]\left(\frac{p(t-s)}{\alpha}\right)^{\beta}\nonumber\\
 &=q\mathbb{E}[\Gamma^{\beta}(s)]+d_{2}\mathbb{E}[\Gamma^{2\beta}(s)]-q^{2}\mathbb{E}[\Gamma^{\beta}(s)]\left(\left(\frac{pt}{\alpha}\right)^{\beta}-\left(\frac{pt-ps}{\alpha}\right)^{\beta}\right)\nonumber\\
 &\sim q\mathbb{E}[\Gamma^{\beta}(s)]+d_{2}\mathbb{E}[\Gamma^{2\beta}(s)]+O(t^{\beta-1}).\label{covariance-large-t}
\end{align}
Similarly, from \eqref{var-fnbp-1} and \eqref{stirlings-appx},
\begin{align}
 \text{Var}[Q_{\beta}(t)]&\sim q\left(\frac{pt}{\alpha}\right)^{\beta}-q^{2}\left(\frac{pt}{\alpha}\right)^{2\beta}+d_{1}\left(\frac{pt}{\alpha}\right)^{2\beta}\nonumber\\
 &=t^{2\beta}\left(q\left(\frac{p}{t\alpha}\right)^{\beta}-q^{2}\left(\frac{p}{\alpha}\right)^{2\beta}+d_{1}\left(\frac{p}{\alpha}\right)^{2\beta}\right)\nonumber\\
 &\sim  t^{2\beta}\left(\frac{p}{\alpha}\right)^{2\beta}\left(d_{1}-q^{2}\right)\nonumber\\&=t^{2\beta}d_{3},\label{variance-large-t}
\end{align}
where $d_{3}=\left(p/\alpha\right)^{2\beta}\left(d_{1}-q^{2}\right)$. Thus, from \eqref{covariance-large-t} and \eqref{variance-large-t}, the correlation between $Q_{\beta}(s)$ and $Q_{\beta}(t)$ for large $t>s$, is
  \begin{align}
 \text{Corr}[Q_{\beta}(s),Q_{\beta}(t)]&\sim\frac{q\mathbb{E}[\Gamma^{\beta}(s)]+d_{2}\mathbb{E}[\Gamma^{2\beta}(s)]}{\sqrt{t^{2\beta}d_{3}}\sqrt{\text{Var}[Q_{\beta}(s)]}}\nonumber= t^{-\beta}\left(\frac{q\mathbb{E}[\Gamma^{\beta}(s)]+d_{2}\mathbb{E}[\Gamma^{2\beta}(s)]}{\sqrt{d_{3}\text{Var}[Q_{\beta}(s)]}}\right),\nonumber
\end{align}
\noindent which decays like the power law $t^{-\beta},~0<\beta<1$. Hence, the FNBP exhibits the LRD property.\end{proof}
\end{theorem}
\noindent Let $\delta>0$ be fixed and define the increments
\begin{equation*}
 Q_{\beta}^{\delta}(t)=Q_{\beta}(t+\delta)-Q_{\beta}(t),~~~t\geq0.
\end{equation*}

\begin{theorem}
The increments $\{Q_{\beta}^{\delta}(t)\}_{t\geq0}$ of the FNBP exhibits the SRD property. 
\end{theorem}
\begin{proof}
By Part (i) of Theorem \ref{Theorem-fnbp-dist-properties},
\begin{align}
  \mathbb{E}[Q_{\beta}^{\delta}(t)]=q(\E[\Gamma^{\beta}(t+\delta)]-\E[\Gamma^{\beta}(t)])&\sim q\left(\frac{pt}{\alpha}\right)^{\beta}\left[\left(1+\delta/t\right)^{\beta}-1\right], \text{ for large }t.\label{autocovariance-increments-1}
\end{align}
Also, for $s\leq t$ and using \eqref{bivariate-fnbfp}, we get
\begin{align}
  \mathbb{E}[Q_{\beta}^{\delta}(s)Q_{\beta}^{\delta}(t)]&=\mathbb{E}[Q_{\beta}(s+\delta)Q_{\beta}(t+\delta)]-\mathbb{E}[Q_{\beta}(s+\delta)Q_{\beta}(t)]-\mathbb{E}[Q_{\beta}(s)Q_{\beta}(t+\delta)]\nonumber \\
    &~~~~~~~~~~+\mathbb{E}[Q_{\beta}(s)Q_{\beta}(t)]\nonumber\\
  &=\beta q^{2}\Big( \mathbb{E}[\Gamma^{2\beta}(t+\delta)B(\beta,1+\beta;\Gamma(s+\delta)/\Gamma(t+\delta))]\nonumber\\
  &~~~~~~~~~~~- \mathbb{E}[\Gamma^{2\beta}(t)B(\beta,1+\beta;\Gamma(s+\delta)/\Gamma(t))]\nonumber\\
    &~~~~~~~~~~~- \mathbb{E}[\Gamma^{2\beta}(t+\delta)B(\beta,1+\beta;\Gamma(s)/\Gamma(t+\delta))]\nonumber\\
      &~~~~~~~~~~~+ \mathbb{E}[\Gamma^{2\beta}(t)B(\beta,1+\beta;\Gamma(s)/\Gamma(t))]\bigg)\nonumber\\
    &\sim q^{2}\big(\mathbb{E}[\Gamma^{\beta}(s+\delta)]\mathbb{E}[\Gamma^{\beta}(t-s)]-\mathbb{E}[\Gamma^{\beta}(s+\delta)]\mathbb{E}[\Gamma^{\beta}(t-s-\delta)]\nonumber\\
  &~~~~~~-\mathbb{E}[\Gamma^{\beta}(s)]\mathbb{E}[\Gamma^{\beta}(t-s+\delta)]+\mathbb{E}[\Gamma^{\beta}(s)]\mathbb{E}[\Gamma^{\beta}(t-s)]\big)~(\text{using }\eqref{autocovariance-last-summand-1})\nonumber\\
  &\sim q^{2}\left(\frac{pt}{\alpha}\right)^{\beta}\bigg[\mathbb{E}[\Gamma^{\beta}(s+\delta)]\left(1-\frac{s}{t}\right)^{\beta}-\mathbb{E}[\Gamma^{\beta}(s+\delta)]\left(1-\frac{s+\delta}{t}\right)^{\beta}\nonumber\\
  &~~~~~~~~~~-\mathbb{E}[\Gamma^{\beta}(s)]\left(1-\frac{s-\delta}{t}\right)^{\beta}+\mathbb{E}[\Gamma^{\beta}(s)]\left(1-\frac{s}{t}\right)^{\beta}\bigg],\label{autocovariance-increments}
  \end{align}
using \eqref{stirlings-appx}. Also from \eqref{autocovariance-increments-1},
\begin{equation*}
\E[Q_{\beta}^{\delta}(s)]\E[Q_{\beta}^{\delta}(t)]\sim q^{2}\delta\left(\frac{pt}{\alpha}\right)^{\beta}\left(\E[\Gamma^{\beta}(s+\delta)]-\E[\Gamma^{\beta}(s)]\right)\left(\left(1+\frac{\delta}{t}\right)^{\beta}-1\right).
\end{equation*}
 Using $(1\pm s/t)^{\beta}\sim 1\pm\beta s/t+\beta(\beta-1)s^{2}/2t^{2},$ for large $t$, in \eqref{autocovariance-increments} and after some tedious calculations, we get
    \begin{align}
    \text{Cov}[Q_{\beta}^{\delta}(s),Q_{\beta}^{\delta}(t)]&\sim q^{2}\left(\frac{pt}{\alpha}\right)^{\beta}\frac{\beta(\beta-1)}{2t^{2}}\Big[(s^{2}-(s+\delta)^{2}-\delta^{2})\E[\Gamma^{\beta}(s+\delta)]\nonumber\\
    &~~~~~~~~~+(s^{2}+\delta^{2}-(s-\delta)^{2})\E[\Gamma^{\beta}(s)]\Big]\nonumber\\
    &=t^{\beta-2}q^{2}\delta\left(\frac{p}{\alpha}\right)^{\beta}\beta(1-\beta)\Big((s+\delta)\mathbb{E}[\Gamma^{\beta}(s+\delta)]-s\mathbb{E}[\Gamma^{\beta}(s)]\Big).\label{fnbp-increments-cov}
    \end{align}
Using $\E[Q^{2}_{\beta}(t)]=q\E[\Gamma^{\beta}(t)]+d_{1}\E[\Gamma^{2\beta}(t)]$ (see \eqref{var-fnbp-1}), we get
  \begin{align}
  \E[(Q_{\beta}^{\delta}(t))^{2}]&=\E[Q_{\beta}^{2}(t+\delta)]+\E[Q_{\beta}^{2}(t)]-2\E[Q_{\beta}(t+\delta)Q_{\beta}(t)]\nonumber\\
  &=q\left(\E[\Gamma^{\beta}(t+\delta)]-\E[\Gamma^{\beta}(t)]\right)+d_{1}\left(\E[\Gamma^{2\beta}(t)]+\E[\Gamma^{2\beta}(t+\delta)]\right)\nonumber\\
  &~~~~~-2d_{2}\E[\Gamma^{2\beta}(t)]-2\beta q^{2}\E[\Gamma^{2\beta}(t+\delta)B(\beta,1+\beta;\Gamma(t)/\Gamma(t+\delta))]\nonumber ~~(\text{using }\eqref{bivariate-fnbfp})\\
  &\sim q\left(\E[\Gamma^{\beta}(t+\delta)]-\E[\Gamma^{\beta}(t)]\right)+d_{1}\left(\E[\Gamma^{2\beta}(t)]+\E[\Gamma^{2\beta}(t+\delta)]\right)\nonumber\\
    &~~~~~-2d_{2}\big(\E[\Gamma^{2\beta}(t)]+\E[\Gamma^{2\beta}(t+\delta)]\big)~~(B(a,b;x)\sim B(a,b),~\text{for $x$ near 1})\nonumber \\
      &=q\left(\E[\Gamma^{\beta}(t+\delta)]-\E[\Gamma^{\beta}(t)]\right)~~(\text{since }d_{1}-2d_{2}=0 )\nonumber\\
      &=q\left(\frac{pt}{\alpha}\right)^{\beta}\left[(1+\delta/t)^{\beta}-1\right]\nonumber~~~(\text{using }\eqref{autocovariance-increments-1})\\
&\sim t^{\beta-1} \beta \delta q \left(\frac{p}{\alpha}\right)^{\beta}\label{var-increments-1}.
\end{align}
From \eqref{autocovariance-increments-1}, we have
\begin{align}
\left(\E[Q_{\beta}^{\delta}(t)]\right)^{2}& \sim \left(\frac{\beta\delta q p^{\beta}}{\alpha^{\beta}}\right)^{2}t^{2(\beta-1)}.\label{var-increments-2}
\end{align}
From \eqref{var-increments-1} and \eqref{var-increments-2}, we get
\begin{align}
   \text{Var}[Q_{\beta}^{\delta}(t)]&\sim t^{\beta-1}\beta\delta q\left(\frac{p}{\alpha}\right)^{\beta}-\left(\frac{\beta\delta q p^{\beta}}{\alpha^{\beta}}\right)^{2}t^{2(\beta-1)}\nonumber\\
   &\sim t^{\beta-1}\beta \delta q\left(\frac{p}{\alpha}\right)^{\beta}.\label{fnbp-increments-var}
  \shortintertext{Thus, from \eqref{fnbp-increments-cov} and \eqref{fnbp-increments-var}, we have for fixed $s$ and large $t$,}
  \text{Corr}[Q_{\beta}^{\delta}(s),Q_{\beta}^{\delta}(t)]&\sim t^{-(3-\beta)/2}\left(\frac{q^{2}\beta(1-\beta)\delta\left(\frac{p}{\alpha}\right)^{\beta}\left((s+\delta)\mathbb{E}[\Gamma^{\beta}(s+\delta)]-s\mathbb{E}[\Gamma^{\beta}(s)]\right)}{\sqrt{\text{Var}[Q^{\delta}_{\beta}(s)]\beta \delta q\left(\frac{p}{\alpha}\right)^{\beta}}}\right).\nonumber
  \end{align}
  Since $(3-\beta)/2\in(1,1.5)$, the increments of the FNBP possess the SRD property.
  \end{proof}
\begin{remark}
Since the FPP is non-stationary, the FNBP is also non-stationary. Also, as seen earlier, the FNBP has the long-range dependence property and its increments are correlated and exhibit the short-range dependence property. Such stochastic models are quite useful for modeling the financial and the time-series data. 
\end{remark}

\noindent Recently, Beghin \cite{beghin8may} and Beghin and Macci \cite{BegClau14} also studied the FNBP. For $0<\beta<1$ and $0<\eta<1$, they define the FNBP as 
\begin{equation*}
 X_{1}(t)=\sum\limits_{i=1}^{N_{\beta}(t,-\ln(1-\eta))}Y_{i} \text{ and } X_{2}(t)=N(\Gamma^{\ast}_{\beta}(t),1), 
\end{equation*}
in \cite{BegClau14} and \cite{beghin8may} respectively, where $Y_{i}$'s are $LS(\eta)$-distributed random variables, independent of $\{N_{\beta}(t,\lambda)\}_{t\geq0}$, $\{\Gamma^{\ast}_{\beta}(t)\}_{t\geq0}=\{\Gamma^{\ast}(E_{\beta}(t))\}_{t\geq0}$ is the fractional gamma process (see \cite{beghin8may}) and $\Gamma^{\ast}(t)\sim G(\alpha,t)$. It is assumed here that $\{\Gamma^{\ast}(t)\}_{t\geq0}$ and $\{E_{\beta}(t)\}_{t\geq0}$ are independent processes. Note when $\beta=1$,
\begin{equation*}
 X_{1}(t)\sim \text{NB}(t,\eta)~~\text{ and }~~X_{2}(t)\sim\text{NB}\left(t,\tfrac{1}{1+\alpha}\right).
\end{equation*}
 \noindent Observe that our definition of the FNBP is
\begin{equation}\label{processours}
 Q_{\beta}(t,\lambda):=N_{\beta}(\Gamma(t),\lambda)\stackrel{d}{=}N(E_{\beta}(\Gamma(t)),\lambda),
\end{equation}
which allows us to compute the one-dimensional distributions having the LRD property. When $\beta=1$, $Q_{1}(t,\lambda)=N(\Gamma(t),\lambda)\sim \text{NB}\left(pt,\frac{\lambda}{\lambda+\alpha}\right)$, $t>0$ (see $e.g.$ \cite{VellaiSree2010}). Let, for $i\geq1,~Y^{\ast}_{i}\stackrel{iid}{\sim} LS\left(\frac{\lambda}{\alpha+\lambda}\right)$. Then it can be seen that
\begin{equation*}
  Q_{1}(t,\lambda)\stackrel{d}{=}\textstyle\sum_{i=1}^{N\left(pt,-\ln\left(\frac{\alpha}{\lambda+\alpha}\right)\right)}Y^{\ast}_{i},
\end{equation*}
where $\{Y^{\ast}_{i}\}_{i\geq1}$ and $\{N(t,\lambda)\}_{t\geq0}$ are independent.
 The following result shows that our process is different from theirs.
\begin{lemma}
 Let $0<\beta<1.$ Then the processes $\{N(\Gamma(E_{\beta}(t)),\lambda)\}_{t\geq0}$ and $\{Q_{\beta}(t,\lambda)\}_{t\geq0}$ are different.
\begin{proof}It is sufficient to prove that the one-dimensional distributions of the processes $\{E_{\beta}(\Gamma(t))\}_{t\geq0}$ and $\{\Gamma(E_{\beta}(t))\}_{t\geq0}$ are different. To see this, let $p(x,t)$ be the {\it pdf} of $\Gamma(E_{\beta}(t))$,$~q(x,t)$ be the {\it pdf} of $E_{\beta}(\Gamma(t))$ and $h_{_{\beta}}(x,t)$ be the {\it pdf} of $E_{\beta}(t)$.
 Now, the Laplace transform (LT) of $p(x,t)$ in the space variable $x$ is
 \begin{align}\label{gammaet}
  \tilde{p}(s,t)=\int_{0}^{\infty}e^{-sx}p(x,t)dx&=\int_{0}^{\infty}\int_{0}^{\infty}e^{-sx}g(x|\alpha,py)h_{_{\beta}}(y,t)dydx\nno\\
  &=\int_{0}^{\infty}\bigg(\frac{\alpha}{\alpha+s}\bigg)^{py}h_{_{\beta}}(y,t)dy=\E\bigg[\bigg(\frac{\alpha}{\alpha+s}\bigg)^{pE_{\beta}(t)}\bigg].
 \end{align}
 It is known (see \cite{MeerStrak13,bingham71}) that $\E[e^{-sE_{\beta}(t)}]=L_{\beta}[-st^{\beta}]$ so that, with $u=e^{-s}$, $\tilde{p}(s,t)=\E[u^{E_{\beta}(t)}]=L_{\beta}[t^{\beta}\log u]$, where $L_{\beta}(z)$ is the Mittag-Leffler function defined in \eqref{Mittag-Leffler-function}. Hence, 
 \begin{equation}\label{gammaet1}
\mathbb{E}[e^{-s\Gamma(E_{\beta}(t))}]= L_{\beta}\left[t^{\beta} p\log\left(\alpha/(\alpha+s)\right)\right]=\sum\limits_{k=0}^{\infty}\frac{(t^{\beta} p\log(\frac{\alpha}{\alpha+s}))^{k}}{\Gamma(\beta k +1)}. 
 \end{equation}
 Similarly, the LT of $q(x,t)$ with respect to variable $x$ is
\begin{align}\label{etgamma}
 \tilde{q}(s,t)&=\int_{0}^{\infty}e^{-sx}q(x,t)dx=\int_{0}^{\infty}\int_{0}^{\infty}e^{-sx}h_{_{\beta}}(x,y)g(y|\alpha,pt)dydx\nno\\
 &=\int_{0}^{\infty}L_{\beta}(-sy^{\beta})g(y|\alpha,pt)dy=\sum\limits_{k=0}^{\infty}\frac{\Gamma(\beta k+pt)}{\Gamma(pt)\Gamma(1+\beta k)}\bigg(\frac{-s}{\alpha^{\beta}}\bigg)^{k}\nno\\
 &=\frac{1}{\Gamma(pt)}{}_{_{2}}\psi_{_{1}}\bigg[\frac{-s}{\alpha^{\beta}}\bigg|\begin{matrix}
(pt,\beta),&(1,1)\\ (1,\beta)&
\end{matrix}\bigg].
\end{align}
 It can be seen that \eqref{gammaet1} and \eqref{etgamma} are different. For example, taking $\beta=1/2,\alpha=2,\lambda=1,p=1,s=1$ and $t=1$, the series in the right-hand side of \eqref{gammaet1} reduces to
 \begin{equation}\label{pxtnum}
  \sum\limits_{k=0}^{\infty}\frac{(\log 2/3)^{k}}{\Gamma(k/2+1)}=-e^{2\log(3/2)}(-1+\text{Erf}[\log(3/2)] ),
 \end{equation}
where $\text{Erf}(z)=\frac{2}{\sqrt{\pi}}\int_{0}^{z}e^{-t^{2}}dt$. But, the right-hand side of \eqref{etgamma} reduces to
\begin{equation}\label{qxtnum}
 \sum_{k=0}^{\infty}\bigg(-\frac{1}{\sqrt{2}}\bigg)^{k}=2-\sqrt{2}.
\end{equation}
Clearly, \eqref{pxtnum} and \eqref{qxtnum} are different, which proves the result.
\end{proof}
\begin{remark}(i) As suggested by a referee, a simple approach would be to compare the mean functions of the processes $\{\Gamma(E_{\beta}(t))\}_{t\geq0}$ and $\{E_{\beta}(\Gamma(t))\}_{t\geq0}$. These  are respectively given by
 \begin{equation*}
  \mathbb{E}[\Gamma(E_{\beta}(t))]=\frac{pt^{\beta}}{\alpha\Gamma(1+\beta)}~~~\text{and}~~~\mathbb{E}[E_{\beta}(\Gamma(t))]=\frac{\Gamma(pt+\beta)}{\alpha^{\beta}\Gamma(1+\beta)\Gamma(pt)},
 \end{equation*}
which follow using $\mathbb{E}[E_{\beta}(t)]=t^{\beta}/\Gamma(1+\beta)$. It is now  clear that the processes are different.\\
(ii) Also, from \eqref{etgamma}, we obtain the LT of $\{Q_{\beta}(t,\lambda)\}_{t\geq0}$ as
\begin{align*}
 \E[e^{-sQ_{\beta}(t,\lambda)}]&=\E\left[\E\left[e^{-\lambda E_{\beta}(\Gamma(t))(1-e^{-s})}|E_{\beta}(\Gamma(t))\right]\right]\\
 &=\frac{1}{\Gamma(pt)}{}_{_{2}}\psi_{_{1}}\bigg[\frac{\lambda(e^{-s}-1)}{\alpha^{\beta}}\bigg|\begin{matrix}
(pt,\beta),&(1,1)\\ (1,\beta)&
\end{matrix}\bigg].
\end{align*}

\end{remark}

 \end{lemma}
\subsection{Connections to \textbf{\textit{pde's}}}
In this section, we discuss $pde$'$s$ governed by the one-dimensional distributions of the FNBP.
\begin{theorem}
 Let $r\in \Z_{+}\backslash \{0\}$. The {\it pmf} \eqref{fnbpmf} of the FNBP solves the following $pde$:

\begin{equation}\label{thm2}
 \frac{\partial^{r}}{\partial\lambda^{r}}\delta_{_{\beta}}(n|\alpha,pt,\lambda)=\frac{1}{\alpha^{n\beta  }\Gamma{(pt)}n!}\sum_{i=0}^{r}\dbinom{r}{i}\dbinom{n}{i}(-1)^{r-i}\lambda^{n-r}H_{2,2}^{1,2}\bigg[\frac{\lambda}{\alpha^{\beta}}\bigg|\begin{matrix}
(-n,1),&(1-n\beta-pt,\beta)\\ (r-i,1),&(-n\beta,\beta)
\end{matrix}\bigg],
\end{equation}
with \begin{equation}\label{intialcon}
\delta_{_{\beta}}(n|\alpha,0,\lambda)= \begin{cases} 
     \hfill 1, \hfill  &  n=0 , \\
      \hfill 0, \hfill & n\geq1.
  \end{cases} \text{ and } \delta_{_{\beta}}(n|\alpha,pt,\lambda)=0,\,\forall\, n<0.
\end{equation}
\end{theorem}
\normalsize
\begin{proof} Note first that the $H$-function defined in \eqref{mellin-H} satisfies (see \cite[Section (1.4.1)]{matsax}), for $r\in\Z_{+}\backslash\{0\}$,

\begin{align*}
\dfrac{\partial^{r}}{\partial z^{r}}\bigg\{z^{-(\gamma \beta_{1}/B_{1})}&H_{p,q}^{m,n}\bigg[z^{\gamma}\bigg|\begin{matrix}
(\alpha_{i},A_{i})_{(1,p)}\\ (\beta_{j},B_{j})_{(1,q)}
\end{matrix}\bigg]\bigg\}=\bigg(\frac{-\gamma}{B_{1}}\bigg)^{r}z^{-r-\left(\frac{\gamma \beta_{1}}{B_{1}}\right)}H_{p,q}^{m,n}\bigg[z^{\gamma}\bigg|\begin{matrix}
(\alpha_{i},A_{i})_{(1,p)}\\ (r+\beta_{1},B_{1}),(\beta_{j},B_{j})_{(2,q)}
\end{matrix}\bigg].
\end{align*}
Taking $p=2,~q=2,~m=1,~n=2,~\alpha_{1}=-n,~\alpha_{2}=1-n\beta-pt$, $\beta_{1}=0,~\beta_{2}=-n\beta$, $A_{1}=B_{1}=1,~A_{2}=B_{2}=\beta$ and $\gamma=1$,  we get
\small
\begin{equation}\label{thm3}\resizebox{.9\hsize}{!}{$
\dfrac{\partial^{r}}{\partial z^{r}}H_{2,2}^{1,2}\bigg[z\bigg|\begin{matrix}
(-n,1),&(1-n\beta-pt,\beta)\\ (0,1),&(-n\beta,\beta)
\end{matrix}\bigg]=(-1)^{r}z^{-r}H_{2,2}^{1,2}\bigg[z\bigg|\begin{matrix}
(-n,1),&(1-n\beta-pt,\beta)\\ (r,1),&(-n\beta,\beta)
\end{matrix}\bigg]$}.
\end{equation}
\normalsize
Now, differentiate $r$ times the right-hand side of \eqref{fnbpmfh} with respect to $\lambda$, use \eqref{thm3} and the Leibniz rule
\begin{equation}\label{diffthm}
 \frac{d^{r}}{dx^{r}}[u(x)v(x)]=\sum_{i=0}^{r}\dbinom{r}{i}\frac{d^{i}}{dx^{i}}(u(x))\frac{d^{r-i}}{dx^{r-i}}(v(x)),
\end{equation}
to obtain the result in \eqref{thm2}.
\end{proof}
\begin{remark}
 When $r=1$, we get
 \begin{equation*}
 \frac{\partial}{\partial\lambda}\delta_{_{\beta}}(n|\alpha,pt,\lambda)=\frac{n}{\lambda}\delta_{_{\beta}}(n|\alpha,pt,\lambda)-\frac{1}{\lambda\Gamma{(pt)}n!}\bigg(\frac{\lambda}{\alpha^{\beta}}\bigg)^{n}H_{2,2}^{1,2}\bigg[\frac{\lambda}{\alpha^{\beta}}\bigg|\begin{matrix}
(-n,1),&(1-n\beta-pt,\beta)\\ (1,1),&(-n\beta,\beta)
\end{matrix}\bigg],
\end{equation*}
with the initial condition given in \eqref{intialcon}.
\end{remark}

\noindent Next, we obtain the fractional $pde$ in the time variable $t$ solved by the FNBP distributions. 
\begin{lemma}\label{rlfdegamma}
 The density of the gamma process $\Gamma(t)\sim G(\alpha,pt)$, given in \eqref{gammaden}, satisfies the following fractional differential equation for any $\nu\geq0$:
 \begin{align}\label{rlfdeeq}
  \partial_{t}^{\nu}g(y|\alpha,pt)&=p\partial_{t}^{\nu-1}\big(\log (\alpha y)-\psi(pt)\big)g(y|\alpha,pt),\,\,\, y>0,\\
  g(y|\alpha,0)&=0\nno,
 \end{align}
 where $\psi(x):=\Gamma'(x)/\Gamma(x)$ is the digamma function and $\partial_{t}^{\nu}$ is the R-L derivative defined in \eqref{rld}.
 \begin{proof} Note first that (see \cite[eq. (3.6)]{spfuntemme})
 \begin{equation}\label{gammaeuler}
  \frac{1}{\Gamma(t)}=\frac{1}{2\pi i}\int_{\bf C}e^{z}z^{-t}dz,
 \end{equation}
where {\bf C} is the Hankel contour given below: \\
\begin{pspicture}(-4,-3.2)(0,3.5)
\psset{xunit=2cm, yunit=2cm, fillcolor=yellow}
\psline[linewidth=1pt]{<->}(-0.25,0)(3,0)
\psline[linewidth=1pt]{<->}(1.5,1.5)(1.5,-1.5)
\psarc[linewidth=1pt, linecolor=red, fillcolor=gray](1.5,0){0.8}{-150}{150}
\rput[tl](0.81,-0.3){{\bf C}}
\rput[tl](1.51,0.25){{ \tiny r}}
\rput[tl](0.51,0.15){{ \tiny $2\delta$}}
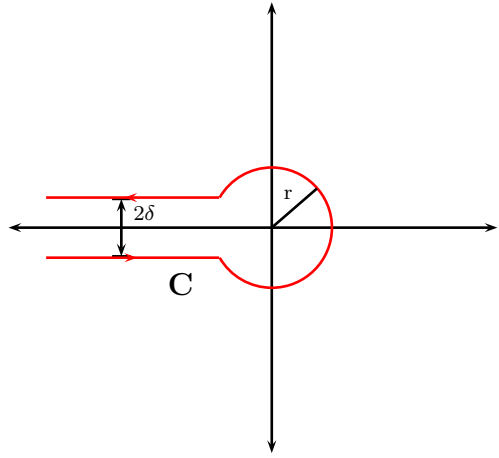
\rput[tl](0.61,-1.65){{Fig. Hankel Contour}}
\psline[linewidth=1pt, fillcolor=gray](1.8,0.26)(1.5,0)
\psline[linewidth=1pt, fillcolor=gray]{|<->|}(0.5,0.2)(0.5,-0.2)
\psline[linewidth=1pt, linecolor=red, fillcolor=gray, ArrowInside=-<](0,0.2) (1.15,0.2)
\psline[linewidth=1pt, linecolor=red, fillcolor=gray,ArrowInside=->](0,-0.2) (1.15,-0.2)
\end{pspicture}
\\\\
Let $\Gamma(a,x)=\int_{x}^{\infty}e^{-t}t^{a-1}dt,~ a>0,x>0,$ denote the incomplete gamma function. Then for $m-1<\nu<m,m\in\mathbb{Z}_{+}\backslash\{0\},$
\begin{align}\label{integramathe}
   \int_{0}^{t}\frac{(\alpha y/z)^{ps}}{(t-s)^{\nu+1-m}}ds&=\left(\frac{\alpha y}{z}\right)^{pt}\left(p\log\tfrac{\alpha y}{z}\right)^{\nu-m}\left\{\Gamma(m-\nu)-\Gamma\left(m-\nu,pt\log\tfrac{\alpha y}{z}\right)\right\}
   \end{align}
   and
   \begin{align}\label{diffmathe}
   \frac{d}{dt}\Gamma\left(m-\nu,pt\log\tfrac{\alpha y}{z}\right)&=-p\left(\frac{\alpha y}{z}\right)^{-pt}\left(p\log\tfrac{\alpha y}{z}\right)^{m-1-\nu}\log\left(\tfrac{\alpha y}{z}\right),
  \end{align}
  which can be checked using Mathematica 8.0. 
  Now, by definition,
  \begin{align*}
   \partial_{t}^{\nu}g(y|\alpha,pt)&=\frac{1}{\Gamma(m-\nu)}\frac{d^{m}}{dt^{m}}\displaystyle\int_{0}^{t}\frac{\alpha^{ps}y^{ps-1}e^{-\alpha y}}{\Gamma(ps)(t-s)^{\nu+1-m}}ds,~~m-1<\nu<m\\
   &=  \frac{(ye^{\alpha y})^{-1}}{\Gamma(m-\nu)}\frac{d^{m}}{dt^{m}}\displaystyle\int_{0}^{t}\frac{(\alpha y)^{ps}}{(t-s)^{\nu+1-m}}\left(\frac{1}{2\pi i }\int_{\bf C}e^{z}z^{-ps}dz\right)ds \nno~~~~ \text{       ~~~~(from \eqref{gammaeuler})}\\
   &=  \frac{(ye^{\alpha y})^{-1}}{\Gamma(m-\nu)}\frac{d^{m}}{dt^{m}}\frac{1}{2\pi i }\int_{\bf C}e^{z}\left(\int_{0}^{t}\frac{(\alpha y/z)^{ps}}{(t-s)^{\nu+1-m}}ds\right)dz\nno ~~~~ \text{ (interchanging the order)} \\
   &=  \frac{(ye^{\alpha y})^{-1}}{\Gamma(m-\nu)}\frac{d^{m}}{dt^{m}}\frac{1}{2\pi i }\int_{\bf C}e^{z}\left(\frac{\alpha y}{z}\right)^{pt}\left(p\log\tfrac{\alpha y}{z}\right)^{\nu-m}\nno\\
   &\hspace*{3cm}\times\left\{\Gamma(m-\nu)-\Gamma\left(m-\nu,pt\log\tfrac{\alpha y}{z}\right)\right\}dz\nno~~~~ \text{       ~~~~(from \eqref{integramathe})}\\
   &=p\frac{(ye^{\alpha y})^{-1}}{\Gamma(m-\nu)}\frac{d^{m-1}}{dt^{m-1}}\frac{1}{2\pi i }\int_{\bf C}e^{z}\left(\frac{\alpha y}{z}\right)^{pt}\log\tfrac{\alpha y}{z}\left(p\log\tfrac{\alpha y}{z}\right)^{\nu-m}\nno\\
   &\hspace*{3cm}\times\left\{\Gamma(m-\nu)-\Gamma\left(m-\nu,pt\log\tfrac{\alpha y}{z}\right)\right\}dz\nno\\
   &-\frac{(ye^{\alpha y})^{-1}}{\Gamma(m-\nu)}\frac{d^{m-1}}{dt^{m-1}}\frac{1}{2\pi i }\int_{\bf C}e^{z}\left(\frac{\alpha y}{z}\right)^{pt}\left(p\log\tfrac{\alpha y}{z}\right)^{\nu-m}\nno\\
   &\hspace*{3cm}\times\left\{-p\left(\frac{\alpha y}{z}\right)^{-pt}\left(p\log\tfrac{\alpha y}{z}\right)^{m-1-\nu}\log\left(\tfrac{\alpha y}{z}\right)\right\}dz\nno~~~~ \text{       ~~~~(from \eqref{diffmathe})}\\
   &=p\frac{(ye^{\alpha y})^{-1}}{\Gamma(m-\nu)}\frac{d^{m-1}}{dt^{m-1}}\frac{1}{2\pi i }\int_{\bf C}e^{z}\left(\frac{\alpha y}{z}\right)^{pt}(\log(\alpha y)-\log (z))\left(p\log\tfrac{\alpha y}{z}\right)^{\nu-m}\nno\\
   &\hspace*{0.5cm}\times\left\{\Gamma(m-\nu)-\Gamma\left(m-\nu,pt\log\tfrac{\alpha y}{z}\right)\right\}dz+\frac{(ye^{\alpha y})^{-1}}{\Gamma(m-\nu)}\frac{d^{m-1}}{dt^{m-1}}\frac{1}{2\pi i }\int_{\bf C}e^{z}dz\nno\\
   &=p\log(\alpha y)\frac{(ye^{\alpha y})^{-1}}{\Gamma(m-\nu)}\frac{d^{m-1}}{dt^{m-1}}\frac{1}{2\pi i }\int_{\bf C}e^{z}\left(\frac{\alpha y}{z}\right)^{pt}\left(p\log\tfrac{\alpha y}{z}\right)^{\nu-m}\nno\\&~\hspace*{2cm}\times\left\{\Gamma(m-\nu)-\Gamma\left(m-\nu,pt\log\tfrac{\alpha y}{z}\right)\right\}dz\nno\\
   &-p\frac{(ye^{\alpha y})^{-1}}{\Gamma(m-\nu)}\frac{d^{m-1}}{dt^{m-1}}\frac{1}{2\pi i }\int_{\bf C}e^{z}\left(\frac{\alpha y}{z}\right)^{pt}\log (z)\left(p\log\tfrac{\alpha y}{z}\right)^{\nu-m}\nno\\
   &\hspace*{2cm}\times\left\{\Gamma(m-\nu)-\Gamma\left(m-\nu,pt\log\tfrac{\alpha y}{z}\right)\right\}dz~~~~ \text{       ~~~~($\because~\int_{\bf C}e^{z}dz=0$)}\\
   &=p\log(\alpha y)\partial_{t}^{\nu-1}g(y|\alpha,pt)-\frac{(ye^{\alpha y})^{-1}}{\Gamma(m-\nu)}\frac{d^{m-1}}{dt^{m-1}}\frac{1}{2\pi i }\int_{0}^{t}\frac{\alpha^{ps}y^{ps}}{(t-s)^{\nu+1-m}}\frac{d}{ds}\frac{1}{\Gamma(ps)}ds\nno\\
   &=p\log(\alpha y)\partial_{t}^{\nu-1}g(y|\alpha,pt)-p\partial_{t}^{\nu-1}\left\{g(y|\alpha,pt)\psi(pt)\right\},
  \end{align*}
  which proves the result.
  \end{proof}

\end{lemma}
\noindent The following corollary corresponds to the case $\nu=1$.
\begin{corollary}
  The density \eqref{gammaden} of the gamma process $\Gamma(t)\sim G(\alpha,pt)$ solves the following $pde$, in the time variable (with $g(y|\alpha,0)=0$):
 \begin{equation*}
  \frac{\partial}{\partial t}g(y|\alpha,pt)=p\left(\log(\alpha y)-\psi(pt)\right)g(y|\alpha,pt),\,\,\, y>0,~t>0.
 \end{equation*}
\end{corollary}
 
\begin{theorem}\label{rlfnbpdiff}
 The {\it pmf} \eqref{fnbpmf} of the FNBP solves the following fractional $pde$ in the time variable $t$:
 \begin{align*}
  \frac{1}{p}\partial_{t}^{\nu}\delta_{_{\beta}}(n|\alpha,pt,\lambda)=\partial_{t}^{\nu-1}(\log(\alpha)-\psi(pt))&\delta_{_{\beta}}(n|\alpha,pt,\lambda)\\ +\int_{0}^{\infty}p_{_{\beta}}(n|y,\lambda)&\log(y)\partial_{t}^{\nu-1}g(y|\alpha,pt)dy,~~~t>0,
 \end{align*}
 where $ \delta_{_{\beta}}(n|\alpha,0,\lambda)=1\text{ if }n=0$ and is zero otherwise, and $p_{\beta}(n|t,\lambda)$ denotes the {\it pmf} of the FPP defined in \eqref{fppd}.
 \begin{proof} Let $m-1<\nu<m$, where $m$ is a positive integer. Then
  \begin{align}
\partial_{t}^{\nu}\delta_{_{\beta}}(n|\alpha,pt,\lambda)&=\partial_{t}^{\nu}\int_{0}^{\infty}p_{_{\beta}}(n|y,\lambda)g(y|\alpha,pt)dy\nonumber\\
&=\frac{1}{\Gamma(m-\nu)}\frac{d^{m}}{dt^{m}}\int_{0}^{t}\int_{0}^{\infty}p_{_\beta}(n|y,\lambda)\frac{g(y|\alpha,ps)}{(t-s)^{\nu+1-m}}dyds\nonumber\\
&=\frac{1}{\Gamma(m-\nu)}\int_{0}^{\infty}p_{_\beta}(n|y,\lambda)\frac{d^{m}}{dt^{m}}\int_{0}^{t}\frac{g(y|\alpha,ps)}{(t-s)^{\nu+1-m}}dsdy\label{dct}\\
&=\int_{0}^{\infty}p_{_{\beta}}(n|y,\lambda)\partial_{t}^{\nu}g(y|\alpha,pt)dy.\label{dct-1}
 \end{align}
 The change in the order of integration in \eqref{dct} can be justified, using Fubini-Tonelli theorem, as follows:
 \begin{align*}
\left|\int_{0}^{t}\int_{0}^{\infty}p_{_\beta}(n|y,\lambda)\frac{g(y|\alpha,ps)}{(t-s)^{\nu+1-m}}dyds\right|&\leq\int_{0}^{t}\frac{1}{(t-s)^{\nu+1-m}}\int_{0}^{\infty}\left|p_{_\beta}(n|y,\lambda)g(y|\alpha,ps)\right|dyds\nonumber\\
&\leq \int_{0}^{t}\frac{1}{(t-s)^{\nu+1-m}}ds=\frac{t^{(m - \nu)}}{(m - \nu)}<\infty.
 \end{align*}
The result now follows from \eqref{rlfdeeq} and \eqref{dct-1}.
 \end{proof}
\end{theorem}

\section{Space Fractional Polya Process}\label{secpp}
\subsection{Polya process} 
\noindent Recall that $\{N(t,\lambda)\}_{t\geq0}$ represents a Poisson process with rate $\lambda>0$ and with $\P[N(t,\lambda)=n]=p(n|t,\lambda)$. The Polya process $\{W^{\Gamma}(t)\}_{t\geq0}:=\{N(t,\Gamma)\}_{t\geq0}$ is obtained by replacing $\lambda$ by a gamma $\Gamma$ random variable, with density $g(x|\alpha,p)$ given in \eqref{gammaden}, which is independent of $\{N(t,\lambda)\}_{t\geq0}$. Then the  {\it pmf} $\eta(n|t,\alpha,p) = \P[W^{\Gamma}(t)=n]$ of the Polya process is given by
\begin{equation}\label{fractional-polya}
\eta(n|t,\alpha,p) = \int_{0}^{\infty}p(n|t,x)g(x|\alpha,p)dx = \frac{t^n}{n!}\frac{\Gamma(n+p)}{\Gamma(p)}\frac{\alpha^p}{(t+\alpha)^{p+n}},
\end{equation}  
which is the {\it pmf} of NB$(p,\frac{t}{\alpha+t})$. 
Since the {\it pmf} $p(n|t,\lambda)$ of the Poisson process satisfies 
\begin{equation*}
 \frac{\partial}{\partial t}p(n|t,\lambda) = -\lambda[p(n|t,\lambda)-p(n-1|t,\lambda)],
\end{equation*}
we have
\begin{equation}\label{eq1.4}
 \frac{\partial}{\partial t}\eta(n|t,\alpha,p) = \int_{0}^{\infty}-x[p(n|t,x)-p(n-1|t,x)]g(x|\alpha,p)dx.
\end{equation}
Now,
\begin{equation}\label{fpp-eq3.6}
 \int_{0}^{\infty}xp(n|t,x)g(x|\alpha,p)dx =\frac{t^{n}\alpha^{p}}{n!\Gamma(p)}\frac{\Gamma(n+p+1)}{(t+\alpha)^{n+p+1}}= \frac{n+p}{t+\alpha}\eta(n|t,\alpha,p),\,\, n\in\mathbb{Z}_{+},
 \end{equation}
 which follows using \eqref{fractional-polya}. Substituting \eqref{fpp-eq3.6} in \eqref{eq1.4}, we obtain
\begin{equation*}
 \frac{\partial}{\partial t}\eta(n|t,\alpha,p) = -\frac{n+p}{t+\alpha}\eta(n|t,\alpha,p) + \frac{n-1+p}{t+\alpha}\eta(n-1|t,\alpha,p),~n\geq 0,
\end{equation*}
with $\eta(n|t,\alpha,p)=0$ for $n<0$, which is the underlying difference-differential equation satisfied by the Polya process.\\
Observe that the negative binomial process $\{Q(t,\lambda)\}_{t\geq0}$ is a L\'evy process (see \cite{kozubo,fell}) so that it has independent increments. However, the Polya process $\{W^{\Gamma}(t)\}_{t\geq0}$ is not a L\'evy process, as it does not have independent increments (see Remark \ref{iipolya}).

\subsection{Space fractional Polya process}

The space fractional Poisson process  $\{\widetilde{N}_{\beta}(t,\lambda)\}_{t\geq0}$, which is a generalization of the Poisson process $\{N(t,\lambda)\}_{t\geq0}$, defined in \cite{sfpp}, can be viewed as (see \cite[Remark 2.3]{sfpp})
$$\widetilde{N}_{\beta}(t,\lambda)\stackrel{d}{=}N(D_{\beta}(t),\lambda),$$
where $D_{\beta}(t)$ is a $\beta$-stable subordinator with $0<\beta<1$.
The {\it pmf} of the space fractional Poisson process is given by (see \cite[Theorem 2.2]{sfpp})
\begin{equation}\label{sfppeq}
 \widetilde{p}_{_{\beta}}(n|t,\lambda)=\P[\widetilde{N}_{\beta}(t,\lambda)=n]=\frac{(-1)^{n}}{n!}\sum\limits_{k=0}^{\infty}\frac{(-\lambda^{\beta}t)^{k}}{k!}\frac{\Gamma(\beta k+1)}{\Gamma(\beta k+1-n)},~t\geq0.
\end{equation}
It solves the fractional difference-differential equation (\cite[eq. (2.4)]{sfpp}) defined by
 \begin{align}\label{orsipolitosfpp}
  \frac{\partial}{\partial t}\widetilde{p}_{_{\beta}}(n|t,\lambda)&=-\lambda^{\beta}(1-B_{n})^{\beta}\widetilde{p}_{_{\beta}}(n|t,\lambda),~~~~\beta\in(0,1],\\
  \widetilde{p}_{_{\beta}}(n|0,\lambda)&=\left\{\begin{array}{ll}1, &\text{for } n=0, \nonumber\\
0, & \mbox{for } n>0,\end{array}
\right.
\end{align}
 
\noindent where $B_{x}$ is the backward shift operator defined by $B_{x}u(x,t)=u(x-1,t)$.
\begin{definition}
 Let $\{D_{\beta}(t)\}_{t\geq0}$ be a $\beta$-stable subordinator and $\Gamma$ be a $G(\alpha,p)$-distributed random variable, independent of $\{D_{\beta}(t)\}_{t\geq0}$ and $\{N(t,\lambda)\}_{t\geq0}$. Then the space fractional Polya process (SFPP) $\{\widetilde{W}^{\Gamma}_{\beta}(t)\}_{t\geq 0}$ is defined as
 \begin{equation*}
 \widetilde{W}^{\Gamma}_{\beta}(t):=\widetilde{N}_{\beta}(t,\Gamma)=N(D_{\beta}(t),\Gamma).
\end{equation*}
\end{definition}

 \begin{theorem}
Let $0<\beta\leq1$. The one-dimensional distributions of the SFPP are 
\begin{equation}\label{fracpolya}
  \widetilde{\eta}_{_{\beta}}(n|t,\alpha,p)=\P[\widetilde{W}_{\beta}^{\Gamma}(t)=n]=\frac{1}{\Gamma(p)}\frac{(-1)^{n}}{n!}{}_{_{2}}\psi_{_{1}}\bigg[-\frac{t}{\alpha^{\beta}}\bigg|\begin{matrix}
(1,\beta),&(p,\beta)\\ (1-n,\beta)
\end{matrix}\bigg],~n\in\Z_{+},
\end{equation}
where ${}_{2}\psi_{1}$ denotes the generalized Wright function defined in \eqref{psifunction}.
\begin{proof}Observe that
\begin{align}
 \widetilde{\eta}_{_{\beta}}(n|t,\alpha,p)&=\E[\P[\widetilde{W}_{\beta}^{\Gamma}(t)=n|\Gamma]]=\int_{0}^{\infty}\widetilde{p}_{_{\beta}}(n|t,y)g(y|\alpha,p)dy \label{fracpolya1}  \\ 
 &=\int_{0}^{\infty}\frac{(-1)^{n}}{n!}\sum\limits_{k=0}^{\infty}\frac{(-y^{\beta}t)^{k}}{k!}\frac{\Gamma(\beta k+1)}{\Gamma(\beta k+1-n)}\frac{\alpha^{p}}{\Gamma(p)}y^{p-1}e^{-\alpha y}dy\nonumber \\
 &=\frac{(-1)^{n}}{n!}\sum\limits_{k=0}^{\infty}\frac{(-t)^{k}}{k!}\frac{\Gamma(\beta k+1)}{\Gamma(\beta k+1-n)}\frac{\alpha^{p}}{\Gamma(p)}\int_{0}^{\infty}y^{{\beta}k+p-1}e^{-\alpha y}dy \nonumber \\
 &=\frac{1}{\Gamma(p)}\frac{(-1)^{n}}{n!}\sum\limits_{k=0}^{\infty}\frac{\Gamma(\beta k+1)\Gamma(\beta k+p)}{\Gamma(\beta k+1-n)}\frac{(-t/\alpha^{\beta})^{k}}{k!}\nonumber \\
 &=\frac{1}{\Gamma(p)}\frac{(-1)^{n}}{n!}{}_{_{2}}\psi_{_{1}}\bigg[-\frac{t}{\alpha^{\beta}}\bigg|\begin{matrix}
(1,\beta),&(p,\beta)\\ (1-n,\beta)
\end{matrix}\bigg].&\qedhere
\end{align}
\end{proof}
 \end{theorem}
When $\beta=1$, we get
\begin{align*}
 \widetilde{\eta}_{1}(n|t,\alpha,p)&=\frac{1}{\Gamma(p)}\frac{(-1)^{n}}{n!}\sum\limits_{k=0}^{\infty}\frac{\Gamma( k+1)\Gamma( k+p)}{\Gamma(k+1-n)}\frac{(-t/\alpha)^{k}}{k!} \\
 &=\frac{1}{\Gamma(p)}\frac{(-1)^{n}}{n!}\sum\limits_{k=0}^{\infty}\frac{\Gamma( k+p)}{(k-n)!}\bigg(\frac{-t}{\alpha}\bigg)^{k}\,\,\,\,\ (\text{put }j=k-n)\\ 
 &=\frac{1}{\Gamma(p)}\frac{(-1)^{n}}{n!}\sum\limits_{n+j=0}^{\infty}\frac{\Gamma( n+j+p)}{j!}\bigg(\frac{-t}{\alpha}\bigg)^{n+j}\\
&=\frac{\Gamma(n+p)}{\Gamma(p)\alpha^{n}}\frac{t^{n}}{n!}\sum\limits_{j=0}^{\infty}\frac{\Gamma( n+j+p)}{\Gamma(j+1)\Gamma(n+p)}\bigg(\frac{-t}{\alpha}\bigg)^{j}\\
&=\frac{t^n}{n!}\frac{\Gamma(n+p)}{\Gamma(p)}\frac{\alpha^p}{(t+\alpha)^{p+n}},
\end{align*}
which is the {\it pmf} of $\text{NB}(p,\frac{t}{\alpha+t})$, as expected.
\begin{remark}
 Consider the time fractional generalization of the Polya process, namely, $\{N_{\beta}(t,\Gamma)\}_{t\geq0}$, where $\Gamma$ is independent of $\{N_{\beta}(t,\lambda)\}_{t\geq0}$, would also be of interest. It seems difficult to compute its $pmf$ given by 
 \begin{align}
  \mathbb{P}[N_{\beta}(t,\Gamma)=n]&=\int_{0}^{\infty}p_{_{\beta}}(n|t,x)g(x|\alpha,p)dx.\nonumber
  \end{align}
\end{remark}

\begin{theorem}
 The SFPP $\{\widetilde{W}^{\Gamma}_{\beta}(t)\}_{t\geq0}$ has stationary increments and is stochastically continuous.
\end{theorem}
\begin{proof}
 Consider first the Polya process $\{W^{\Gamma}(t)\}_{t\geq0}$.\\
 (i) Stationary increments: Let $B$ be a Borel set. Then for $t\geq0,s>0$,
  \begin{align*}
   \P[W^{\Gamma}(t+s)-W^{\Gamma}(s)\in B]&=\E\big[\P[N(t+s,\Gamma)-N(s,\Gamma)\in B|\Gamma]\big]\nno\\
   &=\E\big[\P[N(t,\Gamma)\in B|\Gamma]\big]=\P[W^{\Gamma}(t)\in B],
  \end{align*}
  showing that $\{W^{\Gamma}(t)\}_{t\geq0}$ has stationary increments. It is known that the time change of a process with stationary increments by a process with stationary increments has stationary increments (see \cite[Theorem 1.3.25]{appm}). Since $\{D_{\beta}(t)\}_{t\geq0}$ is a L\'evy process, $\{\widetilde{W}^{\Gamma}_{\beta}(t)\}_{t\geq0}$ also has stationary increments.\\
(ii) Stochastic continuity: Note first that for any process $\{X(t)\}_{t\geq0}$ with stationary increments,
 $$\lim_{t\rightarrow s}\P[|X(t)-X(s)|>a]= 0\Rightarrow \lim_{t\rightarrow 0}\P[|X(t)|>a]=0~~\text{for } a>0.$$ Given $\epsilon>0$, choose $\lambda_{0}$ large enough such that $\int_{\lambda_{0}}^{\infty}g(\lambda|\alpha,p)d\lambda<\epsilon/2$ and since the Poisson process $\{N(t,\lambda)\}_{t\geq0}$ is stochastically continuous, we have for the given $\lambda_{0}$ and $a>0$, there exists a $\delta>0$ such that $\P[N(t,\lambda_{0})>a]<\epsilon/2$  for all $t\in(0,\delta)$.
Suppose $W^{\Gamma}(t)=N(t,\Gamma)$ is not stochastically continuous, then there exists a $t_{0}\in(0,\delta)$ such that $\P[W^{\Gamma}(t_{0})>a]\geq\epsilon$. Again, for $t_{0}\in(0,\delta)$,
\begin{align*}
 \P[W^{\Gamma}(t_{0})>a] &=\E[\P[N(t_{0},\Gamma)>a|\Gamma]]=\int_{0}^{\infty}\mathbb{P}[N(t_{0},\lambda)>a]g(\lambda|\alpha,p)d\lambda\\ 
 &=\int_{0}^{\lambda_{0}}\P[N(t_{0},\lambda)>a]g(\lambda|\alpha,p)d\lambda+\int_{\lambda_{0}}^{\infty}\P[N(t_{0},\lambda)>a]g(\lambda|\alpha,p)d\lambda.
 \intertext{Since $\mathbb{P}[N(t_{0},\lambda)>a]$ is an increasing function of $\lambda,$ we have}
 \mathbb{P}[N(t_{0},\lambda)>a]&<\mathbb{P}[N(t_{0},\lambda_{0})>a]<\epsilon/2~\text{for }0<\lambda\leq\lambda_{0}.\nonumber
 \shortintertext{Hence,}
 \P[W^{\Gamma}(t_{0})>a]&<\frac{\epsilon}{2}\int_{0}^{\lambda_{0}}g(\lambda|\alpha,p)d\lambda+\int_{\lambda_{0}}^{\infty}g(\lambda|\alpha,p)d\lambda\leq\epsilon/2+\epsilon/2= \epsilon,
\end{align*}
which is a contradiction. Hence, $\{W^{\Gamma}(t)\}_{t\geq0}$ is stochastically continuous. Also, by the similar conditioning arguments, it follows that $\{\widetilde{W}^{\Gamma}_{\beta}(t)\}_{t\geq0}$ is also stochastically continuous.
\end{proof}
\begin{remark}\label{iipolya}
The Polya process $\{W^{\Gamma}(t)\}_{t\geq0}$ and the SFPP $\{\widetilde{W}_{\beta}^{\Gamma}(t)\}_{t\geq0}$ are not L\'evy processes, since they do not have independent increments. To see this, let $0\leq t_{1}<t_{2}<t_{3}<\infty$ and $B_{1},B_{2}$ be Borel sets. Then
\begin{align}
\P[W^{\Gamma}(t_{2})-&W^{\Gamma}(t_{1})\in B_{1};W^{\Gamma}(t_{3})-W^{\Gamma}(t_{2})\in B_{2}]\nno\\
&=\E\Big[\P\big[N(t_{2},\Gamma)-N(t_{1},\Gamma)\in B_{1};N(t_{3},\Gamma)-N(t_{2},\Gamma)\in B_{2}|\Gamma\big]\Big]\nno\\
 &=\E\Big[\P\big[N(t_{2},\Gamma)-N(t_{1},\Gamma)\in B_{1}|\Gamma\big]\P\big[N(t_{3},\Gamma)-N(t_{2},\Gamma)\in B_{2}|\Gamma\big]\Big]\label{iipp1}\\
 &\neq\E\Big[\P\big[N(t_{2},\Gamma)-N(t_{1},\Gamma)\in B_{1}|\Gamma\big]\Big]\E\Big[\P\big[N(t_{3},\Gamma)-N(t_{2},\Gamma)\in B_{2}|\Gamma\big]\Big] \label{iipp2}\\
 &=\P[W^{\Gamma}(t_{2})-W^{\Gamma}(t_{1})\in B_{1}]\P[W^{\Gamma}(t_{3})-W^{\Gamma}(t_{2})\in B_{2}].\nonumber
\end{align}
We next show that the right-hand side of \eqref{iipp1} is not equal to the right-hand side of \eqref{iipp2}. Take for example $t_{1}=1,~t_{2}=2,~t_{3}=3, B_{1}=\{n\},$ and $B_{2}=\{m\}$. Then the  right-hand side of \eqref{iipp1} is
\begin{align}
\E\big[\P[N(1,\Gamma)=n|\Gamma]\P[N(1,\Gamma)=m|\Gamma]\big]&=\E\left[\frac{\Gamma^{n}e^{-\Gamma}}{n!}\frac{\Gamma^{m}e^{-\Gamma}}{m!}\right]\nonumber\\
&=\frac{1}{n!m!}\int_{0}^{\infty}y^{n+m}e^{-2y}g(y|\alpha,p)dy\nonumber\\
&=\frac{1}{n!m!}\frac{\alpha^{p}}{\Gamma(p)}\frac{\Gamma(n+m+p)}{(\alpha+2)^{n+m+p}}.\label{first-polya-increments}
\end{align}
\begin{align}
\shortintertext{Again, the right-hand side of \eqref{iipp2}, }
\E\Big[\P\big[N(1,\Gamma)=n|\Gamma\big]\Big]&\E\Big[\P\big[N(2,\Gamma)=m|\Gamma\big]\Big]=\E\left[\frac{\Gamma^{n}e^{-\Gamma}}{n!}\right]\E\left[\frac{\Gamma^{m}e^{-\Gamma}}{m!}\right]\nonumber\\
&=\frac{1}{n!m!}\frac{\alpha^{2p}}{\Gamma^{2}(p)}\frac{\Gamma(n+p)\Gamma(m+p)}{(\alpha+1)^{n+m+2p}}.\label{second-polya-increments}
\end{align}
It can be seen that the right-hand side of \eqref{first-polya-increments} and \eqref{second-polya-increments} are different. In a similar way, we can also prove that $\{\widetilde{W}^{\Gamma}_{\beta}(t)\}_{t\geq0}$ does not have independent increments.
\end{remark}

\begin{remark}The mean $\E[\widetilde{W}^{\Gamma}_{\beta}(t)]$ is infinite, which can be seen as follows. The {\it pgf} of $\widetilde{W}^{\Gamma}_{\beta}(t)$ is, for $|u|\leq1$, \begin{align*}
  \E[u^{\widetilde{W}^{\Gamma}_{\beta}(t)}]=\int_{0}^{\infty}\E[u^{\tilde{N}_{\beta}(t,\lambda)}]g(\lambda|\alpha,p)d\lambda=\int_{0}^{\infty}e^{\lambda^{\beta}t(1-u)^{\beta}}g(\lambda|\alpha,p)d\lambda,
 \end{align*}
see for example \cite[eq. (2.12)]{sfpp}. Now differentiate both sides with respect to $u$, and let $u\rightarrow1$ to obtain infinity.
\end{remark}

\subsection{Connections to \textbf{\textit{pde's}}}We here discuss some $pde$ connections associated with the distributions of the SFPP.\\
First, we establish a result for the process $\{\widetilde{W}^{\Gamma}_{\beta}(t)\}_{t\geq0}$, similar to \eqref{orsipolitosfpp}.
\begin{theorem}Let $k\in\Z_{+}\backslash\{0\}$. The {\it pmf} \eqref{fracpolya} satisfies the following $pde$ in time variable $t$:
 \begin{flalign*} 
  &&\frac{\partial^{k}}{\partial t^{k}}\widetilde{\eta}_{_{\beta}}(n|t,\alpha,p)&=\bigg(-\frac{(1-B_{n})^{\beta}\Gamma(p+\beta)}{\alpha^{\beta}\Gamma(p)}\bigg)^{k}\widetilde{\eta}_{_{\beta}}(n|t,\alpha,p+k\beta)&
 \end{flalign*} 
  $\text{with }\widetilde{\eta}_{_{\beta}}(n|0,\alpha,p)=1\text{ if }n=0\text{ and }\text{zero otherwise.}$
  \end{theorem}
\begin{proof} Note from \eqref{fracpolya1},
 \begin{align}
   \frac{\partial}{\partial t}\widetilde{\eta}_{_{\beta}}(n|t,\alpha,p)&=\int_{0}^{\infty}\frac{\partial}{\partial t}\widetilde{p}_{_{\beta}}(n|t,y)g(y|\alpha,p)dy\nno\\
   &=\int_{0}^{\infty}-y^{\beta}(1-B_{n})^{\beta}\widetilde{p}_{_{\beta}}(n|t,y)g(y|\alpha,p)dy\,\,\,\,(\text{using \eqref{orsipolitosfpp}})\nno\\
   &=-(1-B_{n})^{\beta}\int_{0}^{\infty}y^{\beta}\widetilde{p}_{_{\beta}}(n|t,y)g(y|\alpha,p)dy\nno\\
   &=-\frac{(1-B_{n})^{\beta}\Gamma(p+\beta)}{\alpha^{\beta}\Gamma(p)}\widetilde{\eta}_{_{\beta}}(n|t,\alpha,p+\beta).\label{fracpolyapmfpde-1}
 \end{align}
 The last step is due to the fact 
\begin{equation}\label{gammaform}
y^{\beta}g(y|\alpha,p)=\frac{\Gamma(p+\beta)}{\alpha^{\beta}\Gamma(p)}g(y|\alpha,p+\beta).
\end{equation}
 Now repeating the above computation $k$ times, we get the desired result.\end{proof}
 \begin{corollary}
 The {\it pgf} $G_{\beta}(u|t,\alpha,p)=\E[u^{\widetilde{W}^{\Gamma}_{\beta}(t)}],~|u|\leq1$,  satisfies the following $k$-th order $pde$:
 \begin{align}
  \frac{\partial^{k}}{\partial t^{k}}G_{\beta}(u|t,\alpha,p)=\bigg(-(1-u)^{\beta}\frac{\Gamma(p+\beta)}{\alpha^{\beta}\Gamma(p)}\bigg)^{k}G_{\beta}(u|t,\alpha,p+k\beta),\end{align}
  where $G_{\beta}(u|0,\alpha,p)=1$, and $k\in\Z_{+}\backslash \{0\}$.
\end{corollary}
\begin{proof} Note that
\begin{equation*}
 (1-B_{n})^{\beta}=\sum_{r=0}^{\infty}\frac{\Gamma(\beta+1)}{\Gamma(r+1)\Gamma(\beta-r+1)}(-1)^{r}B_{n}^{r}.
\end{equation*}
From \eqref{fracpolyapmfpde-1},
\begin{align*}
 \frac{\partial}{\partial t}G_{\beta}(u|t,\alpha,p)&=\frac{\partial}{\partial t}\sum_{n=0}^{\infty}u^{n}\widetilde{\eta}_{_{\beta}}(n|t,\alpha,p)=\sum_{n=0}^{\infty}u^{n}\left(-(1-B_{n})^{\beta}\frac{\Gamma(p+\beta)}{\alpha^{\beta}\Gamma(p)}\right)\widetilde{\eta}_{_{\beta}}(n|t,\alpha,p+\beta)\\
 &=-\frac{\Gamma(p+\beta)}{\alpha^{\beta}\Gamma(p)}\sum_{n=0}^{\infty}u^{n}(1-B_{n})^{\beta}\widetilde{\eta}_{_{\beta}}(n|t,\alpha,p+\beta)\\
 &=-\frac{\Gamma(p+\beta)}{\alpha^{\beta}\Gamma(p)}\sum_{n=0}^{\infty}u^{n}\sum_{r=0}^{\infty}\frac{\Gamma(\beta+1)}{\Gamma(r+1)\Gamma(\beta-r+1)}(-1)^{r}B_{n}^{r}\widetilde{\eta}_{_{\beta}}(n|t,\alpha,p+\beta)\\
  &=-\frac{\Gamma(p+\beta)}{\alpha^{\beta}\Gamma(p)}\sum_{n=0}^{\infty}u^{n}\sum_{r=0}^{n}\frac{\Gamma(\beta+1)}{\Gamma(r+1)\Gamma(\beta-r+1)}(-1)^{r}\widetilde{\eta}_{_{\beta}}(n-r|t,\alpha,p+\beta)\\
 &=-\frac{\Gamma(p+\beta)}{\alpha^{\beta}\Gamma(p)}\sum_{r=0}^{\infty}\frac{\Gamma(\beta+1)}{\Gamma(r+1)\Gamma(\beta-r+1)}(-1)^{r}\sum_{n=r}^{\infty}u^{n}\widetilde{\eta}_{_{\beta}}(n-r|t,\alpha,p+\beta)\\
 &=-\frac{\Gamma(p+\beta)}{\alpha^{\beta}\Gamma(p)}\sum_{r=0}^{\infty}\frac{\Gamma(\beta+1)}{\Gamma(r+1)\Gamma(\beta-r+1)}(-1)^{r}\sum_{n=0}^{\infty}u^{n+r}\widetilde{\eta}_{_{\beta}}(n|t,\alpha,p+\beta)\\
 &=-(1-u)^{\beta}\frac{\Gamma(p+\beta)}{\alpha^{\beta}\Gamma(p)}G_{\beta}(u|t,\alpha,p+\beta).
\end{align*}
\noindent Taking the derivative $k$ times, we get the result.
\end{proof}

 \noindent Finally, we obtain the following result for the variable $p$.
 \begin{theorem}\label{pvarsfpppde}
 The {\it pmf} of the SFPP, given in \eqref{fracpolya}, satisfies the following fractional $pde$:
 \begin{align}\label{rlpdfpolya}
 \partial_{p}^{\nu}\widetilde{\eta}_{_\beta}(n|t,\alpha,p)=\partial_{p}^{\nu-1}(\log(\alpha)-\psi(p))\widetilde{\eta}_{_\beta}(n|t,\alpha,p)+\int_{0}^{\infty}\widetilde{p}_{_\beta}(n|t,\lambda)\log(\lambda)\partial_{p}^{\nu-1}g(\lambda|\alpha,p)d\lambda,
  \end{align}
  where $\widetilde{\eta}_{_{\beta}}(n|0,\alpha,p)=1 \mbox{ if }n=0\mbox{ and zero otherwise} .$
  \begin{proof} Note that
\begin{align*}
 \partial_{p}^{\nu}\widetilde{\eta}_{_{\beta}}(n|t,\alpha,p)=\partial_{p}^{\nu}\int_{0}^{\infty}\widetilde{p}_{_\beta}(n|t,\lambda)g(\lambda|\alpha,p)d\lambda&=\int_{0}^{\infty}\widetilde{p}_{_\beta}(n|t,\lambda)\partial_{p}^{\nu}g(\lambda|\alpha,p)d\lambda.
 \end{align*}
The proof now follows from Lemma \ref{rlfdegamma}.
\end{proof}
  \end{theorem}
 \noindent{\bf Acknowledgements.} The authors are deeply grateful to the referees for their detailed report and numerous critical comments and suggestions which improved the paper significantly, both in the content and the quality of the paper.
\def\cprime{$'$}


\end{document}